\documentclass[english]{amsart}
\usepackage[utf8]{inputenc}

\usepackage{geometry}
\usepackage[english]{babel}
\usepackage{mathrsfs}
\usepackage{amsmath}
\usepackage{amsthm}
\usepackage{amssymb}
\usepackage[all]{xy}
\usepackage[unicode=true,pdfusetitle,
 bookmarks=true,bookmarksnumbered=false,bookmarksopen=false,
 breaklinks=false,pdfborder={0 0 1},backref=false,colorlinks=false]
 {hyperref}
 
\usepackage[backend=biber,style=alphabetic]{biblatex}
\usepackage{csquotes}
\usepackage[capitalise]{cleveref}

\makeatletter
\theoremstyle{plain}
\newtheorem{thm}{\protect\theoremname}[section]
\newtheorem{prop}[thm]{\protect\propositionname}
\newtheorem{lem}[thm]{\protect\lemmaname}
\newtheorem{cor}[thm]{Corollary}
\newtheorem{conjecture}[thm]{\protect\conjecturename}

\theoremstyle{definition}
\newtheorem{defn}[thm]{\protect\definitionname}
\newtheorem{rmk}[thm]{Remark}
\newtheorem{rmks}[thm]{Remarks}

\usepackage{xy}
\usepackage{mathrsfs}
\usepackage{hyperref}
\usepackage{color}
\hypersetup{colorlinks=true,citecolor=blue}

\makeatother

\providecommand{\conjecturename}{Conjecture}
\providecommand{\definitionname}{Definition}
\providecommand{\lemmaname}{Lemma}
\providecommand{\propositionname}{Proposition}
\providecommand{\theoremname}{Theorem}

\usepackage{graphicx}
\usepackage[shortlabels]{enumitem}
\usepackage{tikz}
\usepackage{tikz-cd}

\newcommand{\ri}{\mathcal O} \newcommand{\mm}{\mathfrak m} \newcommand{\Q}{\mathbf Q}
\newcommand{\Z}{\mathbf Z}
\newcommand{\F}{\mathbf F}

\DeclareMathOperator{\Spd}{Spd}

\newcommand{\solid}{{\scalebox{0.5}{$\square$}}}
\DeclareMathOperator{\D}{\mathcal D} 
\DeclareMathOperator{\Dqcohri}{\D^a_\solid}
\newcommand{\DqcohriX}[1]{\Dqcohri(\ri^+_{#1}/\pi)}
\DeclareMathOperator{\Dqcohrishriek}{\D^{a!}_\solid}

\newcommand{\Dcatrepsldasm}[2]{\D^{\sm,a}_\solid(#1, #2)}

\newcommand{\infcatinf}{\mathcal Cat_\infty}

\newcommand{\isom}{\cong}
\newcommand{\from}{\mathbin{\leftarrow}}
\newcommand{\xto}[1]{\mathbin{\xrightarrow{#1}}}   \newcommand{\isoto}{\xto\sim}

\newcommand{\injto}{\mathrel{\hookrightarrow}}
\newcommand{\surjto}{\mathrel{\twoheadrightarrow}}
\newcommand{\comp}{\mathbin{\circ}}

\newcommand{\isect}{\mathbin{\cap}}

\newcommand{\tensor}{\otimes}
\newcommand{\dsum}{\oplus}
\newcommand{\bigdsum}{\bigoplus}

\newcommand{\et}{{\mathrm{\acute{e}t}}}
\newcommand{\oc}{{\mathrm{oc}}}

\DeclareMathOperator{\cts}{\mathcal C}

\DeclareMathOperator{\Ind}{Ind}

 \DeclareMathOperator{\GL}{GL} 
\DeclareMathOperator{\id}{id}

\DeclareMathOperator{\Hom}{Hom}
\DeclareMathOperator{\Ext}{Ext}
\DeclareMathOperator{\Fun}{Fun}

\DeclareMathOperator{\IHom}{\underline{\Hom}}

\DeclareMathOperator{\Tot}{Tot}

\DeclareMathOperator{\vStackspi}{vStack^\sharp_\pi}
\DeclareMathOperator{\vStackspip}{vStack^\sharp_{\pi|p}}
\DeclareMathOperator{\vStacksCoeff}{vStack_\Lambda}

\newcommand{\vsite}{{\mathrm{v}}} \DeclareMathOperator{\Spa}{Spa}

\newcommand{\opp}{{\mathrm{op}}} 
\newcommand{\sm}{{\mathrm{sm}}}

\DeclareMathOperator{\Corr}{Corr}

\DeclareMathOperator{\cd}{cd}

\newlist{defenum}{enumerate}{1}
\setlist[defenum]{label=(\alph*), ref=\thedefn.(\alph*)}
\crefalias{defenumi}{defn}

\newlist{thmenum}{enumerate}{1}
\setlist[thmenum]{label=(\roman*), ref=\thethm.(\roman*)}
\crefalias{thmenumi}{prop}

\newlist{propenum}{enumerate}{1}
\setlist[propenum]{label=(\roman*), ref=\theprop.(\roman*)}
\crefalias{propenumi}{prop}

\newlist{lemenum}{enumerate}{1}
\setlist[lemenum]{label=(\roman*), ref=\thelem.(\roman*)}
\crefalias{lemenumi}{lem}

\newlist{corenum}{enumerate}{1}
\setlist[corenum]{label=(\roman*), ref=\thecor.(\roman*)}
\crefalias{corenumi}{cor}

\newlist{rmksenum}{enumerate}{1}
\setlist[rmksenum]{label=(\roman*), ref=\thermks.(\roman*)}
\crefalias{rmksenumi}{rmks}

\newlist{examplesenum}{enumerate}{1}
\setlist[examplesenum]{label=(\roman*), ref=\theexamples.(\alph*)}
\crefalias{examplesenum}{examples}

\begin{document}

\title{$p$-adic sheaves on classifying stacks, and the $p$-adic Jacquet-Langlands correspondence}

\author{David Hansen and Lucas Mann}

\begin{abstract} We establish several new properties of the $p$-adic Jacquet-Langlands functor defined by Scholze in terms of the cohomology of the Lubin-Tate tower. In particular, we reprove Scholze's basic finiteness theorems, prove a duality theorem,  and show a kind of partial K\"unneth formula. Using these results, we deduce bounds on Gelfand-Kirillov dimension, together with some new vanishing and nonvanishing results.

Our key new tool is the six functor formalism with solid almost $\mathcal{O}^+/p$-coefficients developed recently by the second author \cite{mann-p-adic-6-functors}. One major point of this paper is to extend the domain of validity of the $!$-functor formalism developed in \cite{mann-p-adic-6-functors} to allow certain ``stacky" maps. In the language of this extended formalism, we show that if $G$ is a $p$-adic Lie group, the structure map of the classifying small v-stack $B\underline{G}$ is $p$-cohomologically smooth.
\end{abstract}

\maketitle

\tableofcontents{}

\section{Introduction}

\subsection{Prologue}

This paper has two goals. Our immediate goal is to prove some new properties of the $p$-adic Jacquet-Langlands functor defined by Scholze. Unlike the classical Jacquet-Langlands correspondence, which is defined by a simple character identity, Scholze's functor is defined inexplicitly, and has turned out to be rather difficult to study. Scholze himself proved a basic finiteness theorem, and established a local-global compatibility result in some cases. We reprove Scholze's finiteness theorem, and prove a duality theorem and a kind of ``partial'' K\"unneth formula. As consequences, we deduce bounds on Gelfand-Kirillov dimension and prove some new vanishing and nonvanishing results for Scholze's functor.

To explain our second goal, let us briefly recall some recent history. The study of the classical local Langlands correspondence has recently been given new life through the geometrization program proposed by Fargues \cite{fargues-overview}. This program has already seen spectacular success, most notably in the work of Fargues-Scholze \cite{FS}. A key technical ingredient in this program is a sufficiently flexible six-functor formalism for $\ell$-adic \'etale cohomology of diamonds and v-stacks. This formalism was developed by Scholze \cite{etale-cohomology-of-diamonds}, and then extended further in  \cite{enhanced-ell-adic-6-functors}.

Our second goal is to lay some technical foundations for the geometrization of the $p$-adic Langlands correspondence. In particular, we extend the $p$-adic six functor formalism developed in \cite{mann-p-adic-6-functors} to allow certain stacky maps, in parallel with the extension provided in \cite{enhanced-ell-adic-6-functors}, and show that classifying stacks of $p$-adic Lie groups fit perfectly into this formalism. This stacky six-functor machinery allows for an elegant formulation of the proof of our main results. With it in hand, one can also begin a meaningful study of $p$-adic sheaves on $\mathrm{Bun}_G$. We will pursue this in a subsequent paper.

\subsection{Context and duality}

Let us begin by recalling the classical Jacquet-Langlands correspondence. Fix a finite extension $F/\mathbf{Q}_p$ and an integer $n>1$. Set $G=\mathrm{GL}_n(F)$, and let $D/F$ be the central division algebra of invariant $1/n$.

\begin{thm}[Jacquet-Langlands, Deligne-Kazhdan-Vign\'eras, Rogawski] There is a natural bijection \[ \mathrm{JL} : \mathrm{Irr}_{L^2}(G) \overset{\sim}{\to} \mathrm{Irr}(D^\times) \]from the irreducible essentially square-integrable complex representations of $G$ to the irreducible representations of $D^\times$, characterized by the requirement that $\chi_{\mathrm{JL}(\pi)}(g')=(-1)^{n-1} \chi_\pi (g)$ for all pairs of matching elliptic elements $g\in G$ and $g'\in D^\times$. 

This bijection is compatible with twisting and passing to contragredient representations.
 
\end{thm}

A very similar statement holds with $\overline{\mathbf{F}_\ell}$-coefficients instead of complex coefficients, cf. work of Dat and Vign\'eras.

In recent years, a $p$-adic version of the Langlands program has emerged, many aspects of which are still conjectural. In the context of this program, a candidate for a $p$-adic Jacquet-Langlands functor has been proposed by Scholze \cite{scholze-LT}. Let us begin by recalling the definition of Scholze's functor.  Set $C=\widehat{\overline{F}}$, and let $\mathcal{M}_{\infty}$ be the infinite level Lubin-Tate
space over $\mathrm{Spa}\,C$. Recall that this is a perfectoid space
with commuting continuous actions of $G$ and $D^\times$, and with a natural
surjective period map $\pi_{\mathrm{GH}}:\mathcal{M}_{\infty}\to\mathbf{P}_{C}^{n-1}$.
The period map $\pi_{\mathrm{GH}}$ is a $G$-torsor, and is $D^\times$-equivariant for a suitable action of $D^\times$ on projective space. 

Using this geometry, Scholze defines an exact functor, denoted $\pi\mapsto\mathcal{F}_{\pi}$,
from the category $\mathrm{Rep}_{\mathbf{F}_{p}}^{\mathrm{sm}}(G)$ of smooth $G$-representations towards the category of  \'etale $\mathbf{F}_{p}$-sheaves on $\mathbf{P}_{C}^{n-1}$.
This functor is defined by a simple and explicit recipe: given any \'etale map
$U\to\mathbf{P}_{C}^{n-1}$, one defines $\mathcal{F}_{\pi}(U)=\mathcal{C}_{G}(|U\times_{\mathbf{P}_{C}^{n-1}}\mathcal{M}_{\infty}|,\pi)$. Since $\pi_{\mathrm{GH}}$ is $D^\times$-equivariant, the sheaf $\mathcal{F}_\pi$ is also $D^\times$-equivariant, and the \'etale cohomology groups
\[\mathcal{J}^{i}(\pi):=H^{i}(\mathbf{P}_{C}^{n-1},\mathcal{F}_{\pi})\]
are naturally smooth $D^\times$-representations which also carry a natural action of the Weil group $W_F$. By general vanishing theorems, it is easy to see that $\mathcal{J}^i(\pi)=0$ for all $i>2n-2$, and it is also easy to compute $\mathcal{J}^0(\pi)$ explicitly.  However, these are the only obvious properties of $\mathcal{J}^i.$ One of the main results of \cite{scholze-LT} is a highly nonobvious finiteness theorem: if $\pi$ is an \emph{admissible} $G$-representation, then also each $\mathcal{J}^i(\pi)$ is an admissible $D^\times$-representation. As a key step in the proof of this finiteness result, Scholze also proved a form of the primitive comparison theorem: if $\pi$ is admissible, the natural map
\[H^{i}(\mathbf{P}_{C}^{n-1},\mathcal{F}_{\pi}) \otimes \mathcal{O}_C/p \to H^{i}(\mathbf{P}_{C}^{n-1},\mathcal{F}_{\pi} \otimes \mathcal{O}^+ /p) \]
is an almost isomorphism.

Consideration of simple examples shows that in general, $\mathcal{J}^i(\pi)$ can be nonzero for multiple values of $i$.\footnote{This already happens for the simplest example $\pi=\mathbf{1}$, in which case $\mathcal{F}_\pi =\F_p$ and $\mathcal{J}^i(\pi)=H^{i}(\mathbf{P}_{C}^{n-1},\F_p)$ is the mod $p$ \'etale cohomology of projective space, which by the usual computation is nonzero in degrees $0,2,4,\dots,2n-2$.} As such, we prefer to package the individual cohomologies $\mathcal{J}^i$ into a single operation on the derived category. More precisely, by the exactness of the association $\pi \mapsto \mathcal{F}_\pi$, it naturally upgrades to a functor $\D(G):=\D(\mathrm{Rep}_{\mathbf{F}_{p}}^{\mathrm{sm}}(G))\to \D(\mathbf{P}^{n-1}_{C,\mathrm{\acute{e}t}},\mathbf{F}_{p})$,
and we write $\mathcal{J}$ for the composite functor
\[
\D(G)\overset{V \mapsto \mathcal{F}_V}{\longrightarrow} \D(\mathbf{P}^{n-1}_{C,\mathrm{\acute{e}t}},\mathbf{F}_p)\overset{R\Gamma}{\to}\D(D^\times).
\]
By the finiteness theorems mentioned above, the functor $\mathcal{J}$ preserves bounded complexes with admissible cohomology.

Now, one would guess that $\mathcal{J}$ intertwines some natural duality operation on $G$-representations and $D^\times$-representations. However, in stark contrast with the case of complex or $\ell$-adic representations, the naive functor of passing to smooth duals on mod-$p$ representations of $p$-adic Lie groups is essentially useless. To remedy this, Kohlhaase has introduced a family of ``derived'' smooth duality functors $\mathcal{S}^i_H$ on the smooth mod-$p$ representations of any $p$-adic Lie group $H$, which can be packaged into a single endofunctor $\mathcal{S}_H: \D(H) \to \D(H)$.  This functor has excellent properties, which we review in \S \ref{smoothreps} below.

Our first main result can now be stated as follows.

\begin{thm}\label{dualitymaintheorem}
There is a natural equivalence $\mathcal{S}_{D^\times}\circ\mathcal{J}\cong(\mathcal{J}\circ\mathcal{S}_{G})[2n-2](n-1)$
of functors $\D_{\mathrm{adm}}(G)\to \D_{\mathrm{adm}}(D^\times)$. 
\end{thm}

In other words, Scholze's functor $\mathcal{J}$ intertwines Kohlhaase's derived duality on $G$-representations and $D^\times$-representations, up to an explicit shift and twist. This result was conjectured by the first author, after observing that for certain $V$'s arising from the completed cohomology of Shimura varieties, both $\mathcal{S}_{D^\times}\circ\mathcal{J}(V)$ and  $(\mathcal{J}\circ\mathcal{S}_{G})(V)[2n-2](n-1)$ can be computed by hand, and turn out to be abstractly isomorphic.

In the remainder of this section, we give a detailed sketch of the proof of Theorem \ref{dualitymaintheorem}. The first key ingredient is a more conceptual interpretation of Kohlhaase's derived duality functor.

\begin{prop}\label{kohlhaaseconceptual}For any $p$-adic Lie group $G$, the functor $\mathcal{S}_G(-): \D(G) \to \D(G)$ is naturally isomorphic to the functor $\IHom_G(-,\mathbf{1})$. Here $\IHom_G$ denotes the natural internal Hom in $\D(G)$.
\end{prop}

Be aware that in this paper, $\IHom$ will always denote the \emph{derived} internal hom (which is often denoted by $R\IHom$ in the literature) -- see \cref{sec:intro.notation} for more remarks on that choice. In the case of mod-$\ell$ coefficients with $\ell\neq p$, $\IHom_G(-,\mathbf{1})$ is actually $t$-exact and therefore restricts to a functor on the hearts, where it coincides with the usual functor of smooth duality (\cref{ellnotpduality}). In this light, Proposition \ref{kohlhaaseconceptual} puts Kohlhaase's duality on the same footing as classical smooth duality. We also note that Proposition \ref{kohlhaaseconceptual} was independently discovered by Schneider-Sorensen \cite{schneider-sorensen-dual-reps}.

Next, we invoke a slightly different formula for the functor $\mathcal{J}$. For this, we briefly return to the geometry of the Lubin-Tate tower. Along with the surjective period map $\pi_{\mathrm{GH}}:\mathcal{M}_{\infty}\to\mathbf{P}_{C}^{n-1}$ mentioned earlier, there is a second period map $\pi_{\mathrm{HT}}:\mathcal{M}_{\infty}\to\Omega_{C}^{n-1}$, presenting $\mathcal{M}_{\infty}$ as a $D^\times$-torsor over Drinfeld space.
Both period maps are $G\times D^\times$ -equivariant,
with $G$ (resp. $D^\times$) acting trivially on $\mathbf{P}_{C}^{n-1}$
(resp. on $\Omega_{C}^{n-1}$). In particular, we can pass to the stack
quotient $X$ of $\mathcal{M}_{\infty}$ by the whole $G\times D^\times$-action,
obtaining a small v-stack $X$ which admits two natural presentations \[X\cong[\mathbf{P}_{C}^{n-1}/\underline{D^\times}]\cong[\Omega_{C}^{n-1}/\underline{G}].\] We think of this stack as sitting in a diagram
\[
\xymatrix{ & X\ar[dl]_{f}\ar[dr]^{g}\\
B\underline{G}=[\mathrm{Spa}\,C/\underline{G}] &  & B\underline{D^\times}=[\mathrm{Spa}\,C/\underline{D^\times}]
}
\]
where $f$ is smooth, and $g$ is smooth and proper, with both maps being of pure relative dimension $n-1$. Let us call this
\textbf{the fundamental diagram}. Now, for any small v-stack $X$, the formalism in \cite{etale-cohomology-of-diamonds} defines an associated category $\mathcal{D}_{\et}(X,\F_p)$ of \'etale $\F_p$-sheaves. We will see below (Proposition \ref{Detclassifying}) that there is a natural identification $\D(H) = \mathcal{D}_{\et}(B\underline{H},\F_p)$ of symmetric monoidal categories for any $p$-adic Lie group $H$. Under this identification, it is not hard to show that $\mathcal{J}$ is simply given by the operation $g_{\et \ast} f_{\et}^{\ast}$ (where by $g_{\et\ast}$ we denote the \emph{derived} pushforward, often denoted $Rg_{\et\ast}$ -- cf. \cref{sec:intro.notation}). At this point, it is heuristically clear how the argument should go: $f$ is smooth and $g$ is proper smooth, so the operation $g_{\et \ast} f_{\et}^{\ast}$ should commute with internal hom towards the monoidal unit, up to an explicit shift and twist. However, it is entirely unclear how to make this heuristic rigorous, because the $!$-functors associated with the maps $f$ and $g$ are conspicuously absent from the $\mathcal{D}_{\et}$-formalism.\footnote{Of course, this absence is a special feature of the $\ell=p$ setting of the present paper; one of the main points of \cite{etale-cohomology-of-diamonds} is to construct the $!$-functors in the $\mathcal{D}_{\et}$-formalism when the coefficient ring is killed by an integer prime to $p$.} 

The key new tool which lets us circumvent this issue is the formalism of \emph{solid almost $\mathcal{O}^+/p$-sheaves} recently developed by the second author \cite{mann-p-adic-6-functors}. This is a theory which attaches to any small v-stack $X$ over $\Spa C$ a symmetric monoidal stable $\infty$-category $\Dqcohri(\mathcal{O}^{+}_X/p)$, whose objects are (to first approximation) complexes of solid quasicoherent almost $\mathcal{O}^{+}_X/p$-sheaves on the \'etale site of $X$. There is also a variant $\Dqcohri(\mathcal{O}^{+}_X/p)^{\varphi}$ where the objects are equipped with a suitable $\varphi$-module structure, which comes with an evident forgetful functor $\Dqcohri(\mathcal{O}^{+}_X/p)^{\varphi} \to \Dqcohri(\mathcal{O}^{+}_X/p)$. Crucially, the sheaf theories $X \mapsto \Dqcohri(\mathcal{O}^{+}_X/p)$ and $\Dqcohri(\mathcal{O}^{+}_X/p)^{\varphi}$ underlie \emph{full} six functor formalisms, and all six functors are compatible with the operation of forgetting the $\varphi$-module structure. In particular, for any map of small v-stacks $f:X \to Y$ over $\Spa C$, we have functors $f_{\ast}$ and $f^{\ast}$, and if $f$ is a so-called bdcs\footnote{``$p$-\textbf{b}ounde\textbf{d}, locally \textbf{c}ompactifiable, and representable in locally \textbf{s}patial diamonds"} map we also have functors $f_!$ and $f^!$ with all expected properties. Moreover, there is a notion of $p$-cohomologically smooth maps, for which $f^!$ is an invertible twist of $f^\ast$. Finally, if $H$ is a $p$-adic Lie group, there is a natural identification $\Dqcohri(\mathcal{O}^{+}_{B\underline{H}}/p)= \D^{\sm,a}_\solid(H,\ri_C/p)$, where the right-hand side denotes the derived $\infty$-category of smooth $H$-representations on complexes of solid almost $\mathcal{O}_C / p$-modules (and similarly with $\varphi$-structures).

Returning to the fundamental diagram, we can define a solid variant $\mathcal{J}_{\solid}$ of the functor $\mathcal{J}$ by the formula \[\mathcal{J}_{\solid} = g_{\ast} f^{\ast}: \Dqcohri(\mathcal{O}^{+}_{B\underline{G}}/p)^{\varphi} \to \Dqcohri(\mathcal{O}^{+}_{B\underline{D^\times}}/p)^{\varphi}. \]
Using the formalism developed in \cite{mann-p-adic-6-functors}, it is entirely straightforward to prove a solid variant of our goal, i.e. an isomorphism $\mathcal{S}_{D^\times,\solid} \circ \mathcal{J}_\solid \cong (\mathcal{J}_\solid \circ \mathcal{S}_{G,\solid})[2n-2](n-1)$,
where (for any $p$-adic Lie group $H$) $\mathcal{S}_{H,\solid}$ denotes the internal hom towards the monoidal unit in $\Dqcohri(\mathcal{O}^{+}_{B\underline{H}}/p)^{\varphi}$. However, we have potentially lost sight of our goal of understanding duality for the functor $\mathcal{J}$. 

At this point we use another key feature of the formalism developed in \cite{mann-p-adic-6-functors}, namely the \emph{$p$-torsion Riemann-Hilbert correspondence}, which gives a fully faithful embedding $ - \otimes \mathcal{O}^{+a}_X/p: \D_{\et}(X,\F_p)^{\dagger} \injto \Dqcohri(\mathcal{O}^{+}_X/p)^{\varphi}$ for any small v-stack $X$. Here the superscripted $\dagger$ denotes the restriction to the full subcategory of overconvergent sheaves. This embedding comes with a right adjoint $M\to M^\varphi$ which is (morally) the functor of $\varphi$-invariants. Using the basic properties of these functors, it is formal to check that $\mathcal{J}$ can be recovered from $\mathcal{J}_{\solid}$, in the sense that $\mathcal{J}(A) \cong \mathcal{J}_{\solid}(A \otimes \mathcal{O}^{+a}_{B\underline{G}}/p )^{\varphi}$ for any $A \in \D(G) = \mathcal{D}_{\et}(B\underline{G},\F_p)^{\dagger}$. We now apply a key new result, Proposition \ref{prop:admissiblecriterion}, which gives a practical criterion for an object  $M \in \Dqcohri(\mathcal{O}^{+}_{B\underline{H}}/p)^{\varphi}$ to arise from a bounded admissible complex in $\D(H)$ via the Riemann-Hilbert embedding, and shows that in this case the solid dual $\mathcal{S}_{H,\solid}$ is the Riemann-Hilbert embedding of the smooth dual $\mathcal{S}_H$. It turns out that in order to apply this criterion to $\mathcal{J}_{\solid}$, we only need to know two things:
\begin{enumerate}
    \item $\mathcal{J}_{\solid}$ preserves discrete objects, and
    
    \item for any admissible $A\in \D(G)$, $\mathcal{J}_{\solid}(A \otimes \mathcal{O}^{+a}_{B\underline{G}}/p )$ is isomorphic to its solid double dual.
\end{enumerate}

We now conclude by noting that (1) is formal, and (2) is a consequence of the solid duality theorem. For the details of this argument, see Theorem \ref{solidduality} and its proof.

Let us emphasize that in the course of running this argument, we will reprove Scholze's finiteness and primitive comparison theorems for $\mathcal{J}$. Just as in Scholze's proof, the key point here is the properness of $\mathbf{P}^{n-1}_{C}$, or equivalently the properness of the map $g$. However, Scholze exploits this properness in a very indirect way, by a Cartan-Serre style argument with shrinking covers and an elaborate use of spectral sequences. For us, the properness of $g$ is used entirely through the identification $g_!=g_{\ast}$.

\subsection{The stacky 6-functor formalism} In the process of proving our main results we also extend the 6-functor formalism from \cite{mann-p-adic-6-functors} to ``stacky'' maps and study the behavior of this extended 6-functor formalism in the setting of classifying stacks. More concretely, we say that a map $f\colon Y \to X$ of small v-stacks is \emph{$p$-fine} if it admits a cover $g\colon Z \to Y$ such that $g$ admits universal $p$-codescent and $f \comp g$ is bdcs (see \cref{def:p-fine-maps}). Using the results on abstract 6-functor formalisms in \cite{mann-p-adic-6-functors}, it is formal that the 6-functor formalism extends to $p$-fine maps, i.e. to any $p$-fine map $f$ we get associated functors $f_!$ and $f^!$ with all the expected properties. The typical example of a map satisfying universal $p$-codescent is a $p$-cohomologically smooth cover, but it is important to not restrict to smooth covers in the definition of $p$-fine maps.

Now, if $G$ is a profinite group which has finite $p$-cohomological dimension, we prove that the structure map $f\colon */G \to *$ is $p$-fine and that $f_! = f_*$ (see \cref{rslt:finite-cd-p-implies-classifying-stack-is-pfine}). The key point here is to show that the obvious cover $* \to */G$, which is not $p$-cohomologically smooth, nevertheless admits universal $p$-codescent (see Lemma \ref{rslt:p-codescent-for-classifying-stack}); the argument for this crucially uses Mathew's notion of descendability. This result easily implies that for all \emph{locally} profinite groups $G$ which \emph{locally} have finite $p$-cohomological dimension, the map $*/G \to *$ is $p$-fine. Moreover, we provide a precise criterion on $G$ determining whether or not $*/G \to *$ is $p$-cohomologically smooth (see \cref{rslt:locally-p-Poincare-equiv-p-cohomo-smoothness}). If this is the case then we say that $G$ is \emph{virtually $p$-Poincaré}. Using classical results of Lazard \cite{lazard-monster}, we show that every $p$-adic Lie group $G$ is virtually $p$-Poincaré (see \cref{padicLiePoincare}), so that for these groups the associated map $*/G \to *$ is $p$-cohomologically smooth. This applies to groups like $\GL_n(\Q_p)$ and $D^\times$. With these results in hand, we deduce a $p$-adic version of the $\ell$-adic smoothness criterion \cite[Proposition 24.2]{etale-cohomology-of-diamonds}:

\begin{prop}[see \cref{rslt:p-adic-smoothness-criterion-for-quotient}]
Let $G$ be a locally profinite group which is virtually $p$-Poincaré (e.g. $G$ could be any $p$-adic Lie group). Suppose $G$ acts on a small v-stack $Y$ and $f\colon Y \to X$ is a $G$-equivariant separated bdcs map of small v-stacks, where $G$ acts trivially on $X$. Then the induced map $Y/G \to X$ is $p$-fine, and it is $p$-cohomologically smooth as soon as $f$ is $p$-cohomologically smooth.
\end{prop}

Note that in the $\ell$-adic case, one proves \cite[Proposition 24.2]{etale-cohomology-of-diamonds} directly, and then deduces that classifying stacks of $p$-adic Lie groups are $\ell$-cohomologically smooth. Such a direct argument seems very unlikely in the $p$-adic setting, due to the nonexistence of Haar measures.

The above results on the stacky 6-functor formalism and classifying stacks are useful for proving the crucial \cref{prop:admissiblecriterion}. However, all we really need in that proof is the more classical statement that smooth group cohomology of $p$-adic Lie groups satisfies Poincaré duality (as proved by Lazard). The real power of the stacky 6-functor formalism will reveal itself when we apply it to $\mathrm{Bun}_G$ in subsequent papers.

\subsection{Applications to representation theory}

One of our main motivations for proving the duality theorem is that Kohlhaase's duality functors are intimately related to the \emph{Gelfand-Kirillov dimension} $\dim_G \pi$ of an admissible representation $\pi$. Recall that this is a nonnegative integer measuring the asymptotic size of $\pi$; for instance, $\dim_G \pi = 0$ exactly when $\dim_{\F_p} \pi < \infty$.

Computing the dimension of admissible representations arising in nature is generally quite difficult. However, our duality theorem gives a powerful tool for understanding how the functor $\mathcal{J}$ interacts with Gelfand-Kirillov dimension. Our first application in this direction is the following result, which gives a quantitative form of Scholze's finiteness theorem. 

\begin{thm}\label{diminequality}
Let $\pi$ be an admissible smooth $G=\mathrm{GL}_{n}(F)$-representation,
and let $N_{\pi}\geq0$ be the least integer such that $\mathcal{J}^{i}(\pi)=0$
for all $i>N_{\pi}$. Then
\[
\dim_{D^{\times}}\mathcal{J}^{i}(\pi)\leq\dim_{G}(\pi)+N_{\pi}
\]
for all $i$. In particular, since $N_{\pi}\leq2n-2$, we have
\[
\dim_{D^{\times}}\mathcal{J}^{i}(\pi)\le\dim_{G}(\pi)+2n-2
\]
for all $\pi$ and all $i$.
\end{thm}

Note that the possible increase of dimension by $2n-2$ is quite small compared to the maximal Gelfand-Kirillov dimension in this setting, which is $\dim G = \dim D^\times = n^2[F:\mathbf{Q}_p]$. Nevertheless, we conjecture that Theorem \ref{diminequality} is not optimal, and that in fact one always has an inequality
\[\dim_{D^{\times}}\mathcal{J}^{i}(\pi)\leq\dim_{G}(\pi),\]
cf. Conjecture \ref{diminequalityoptimal} below. This conjecture in turn implies that $\mathcal{J}^{2n-2}(\pi)$ is always a finite-dimensional vector space, which can be computed explicitly (Proposition \ref{Jtopdegree}).

Under a stronger assumption on $\pi$, we can also prove a bound in the other direction, namely that for some $i$ the association $\pi \rightsquigarrow \mathcal{J}^i(\pi)$ cannot decrease the dimension by very much. Here is a precise statement.

\begin{thm}\label{dimlowerbound}Let $\pi$ be an admissible smooth $G$-representation which is Cohen-Macaulay in the sense that $\mathcal{S}_{G}^i(\pi)$ is nonzero for a single value of $i$. Then either $\mathcal{J}(\pi)=0$ or
\[ \max_i \dim_{D^\times} \mathcal{J}^i(\pi) \geq \dim_G \pi - (2n-2).
\]
\end{thm}

We note that the Cohen-Macaulay condition is often satisfied in practice, and is preserved by parabolic induction.

The duality theorem is even more useful when combined with nontrivial vanishing theorems for $\mathcal{J}$. To illustrate this, we now specialize to the case $n=2$, but with $F/\mathbf{Q}_{p}$
still arbitrary. Set $d=[F:\mathbf{Q}_{p}]$. Let $B\subset G=\mathrm{GL}_{2}(F)$
be the upper-triangular Borel subgroup, with $T=F^{\times}\times F^{\times}\subset B$
the diagonal maximal torus.  For any smooth characters $\chi_1,\chi_2:F^\times \to \F_p ^\times$, we have the usual parabolic induction $\pi(\chi_{1},\chi_{2})=\mathrm{Ind}_{B}^{G}(\omega\chi_{1}\otimes\chi_{2})$,
where $\omega$ is the mod-$p$ cyclotomic character. We assume that $\chi_1 / \chi_2 \neq \omega^{\pm 1}$. Then $\pi(\chi_1,\chi_2)$ and $\pi(\chi_2,\chi_1)$ are irreducible, and are exchanged by the functor $\mathcal{S}_G^d(-) \otimes (\chi_1 \chi_2 \omega \circ \det)$. Moreover, $\mathcal{S}_G^i(-)$ applied to either representation vanishes for all $i\neq d$.

By irreducibility, it is easy to see that $\mathcal{J}^0(\pi(\chi_1,\chi_2))=0$, and likewise with the characters exchanged. Much less obviously, the main theorem of \cite{johansson-ludwig} shows also that $\mathcal{J}^2(\pi(\chi_1,\chi_2))=0$. With these vanishing results in mind, we define \[ \tau(\chi_1,\chi_2) = \mathcal{J}^1(\pi(\chi_1,\chi_2)). \]
Our duality theorem then implies the following dichotomy.

\begin{thm}\label{principalseriesdichotomy}Notation and assumptions as above, exactly one of the following is true.

i. $\tau(\chi_1,\chi_2)= \tau(\chi_2,\chi_1)=0.$

ii. Both $\tau(\chi_1,\chi_2)$ and $\tau(\chi_2,\chi_1)$ are nonzero, and are exchanged by the functor $\mathcal{S}_{D^\times}^d(-) \otimes (\chi_1 \chi_2 \omega \circ \mathrm{Nm})$. In particular, they are both of dimension $d$. Moreover, $\mathcal{S}_{D^\times}^i(-)$ applied to either representation vanishes for all $i\neq d$.
\end{thm}

We conjecture that in fact scenario i. in this theorem \emph{never} occurs. When $F=\mathbf{Q}_p$ and $p>3$, this follows from work of Pa\v{s}k\={u}nas \cite{paskunas-consequences}, who showed by global methods that $\tau(\chi_1,\chi_2)$ and $\tau(\chi_2,\chi_1)$ cannot \emph{both} vanish. Combining his results with the previous theorem yields the following corollary.

\begin{cor}Maintain the notation and assumptions above, and assume moreover that $F=\mathbf{Q}_p$ and $p>3$. Then $\tau(\chi_1,\chi_2)$ and $\tau(\chi_2,\chi_1)$ are nonzero, of dimension one, and are exchanged by the functor $\mathcal{S}_{D^\times}^1(-) \otimes (\chi_1 \chi_2 \omega \circ \mathrm{Nm})$. Moreover, $\mathcal{S}_{D^\times}^0(-)$ applied to either representation vanishes.
\end{cor}

We note that the final vanishing statement here has a more concrete interpretation: it simply means that $\tau(\chi_1,\chi_2)$ does not admit any nonzero $H$-stable quotient $\sigma$ with $\dim_{\F_p} \sigma < \infty$, for any open subgroup $H \subset D^\times$ (and similarly for $\tau(\chi_2,\chi_1)$). Intuitively, this means that $\tau(\chi_1,\chi_2)$ is ``maximally non-semisimple". It would be very interesting to find an explicit construction of one-dimensional admissible representations of $D^\times$ with this property. We also note that a similar but more complicated theorem holds in the setting of \emph{reducible} principal series representations of $\mathrm{GL}_2(\mathbf{Q}_p)$; see \S\ref{nongeneric} for details.

Our last result is an intriguing alternative formula for $\mathcal{J}(\pi)$ in the case where $\pi$ is parabolically induced. To state this result, we momentarily change our notation slightly, to emphasize the dependence of the Jacquet-Langlands functor on the chosen integer $n$. More precisely, we write $\mathcal{J}_n$ for the Jacquet-Langlands functor from $\D(\mathrm{GL}_n(F))$ towards $\D(D_{1/n}^{\times})$. Choose an integer $0<d<n$, and let $P_{n-d,d}\subset \mathrm{GL}_n(F)$ be the usual upper-triangular parabolic with Levi $\mathrm{GL}_{n-d}(F)\times \mathrm{GL}_d(F)$. If $\sigma_1 \boxtimes \sigma_2$ is any smooth representation of the Levi, we write $\mathrm{Ind}_{P_{n-d,d}}^{\mathrm{GL}_n(F)}(\sigma_1 \boxtimes \sigma_2)$ for the usual parabolic induction to a smooth representation of $\mathrm{GL}_n(F)$. In \S \ref{partialkunnethsection} below, we define a certain auxiliary spatial diamond $W_{n,d}$ over $\Spd C$, which comes equipped with a natural $D_{1/n}^\times$-action as well as natural $D_{1/n}^\times$-equivariant maps $h_1: W_{n,d} \to B\underline{\mathrm{GL}_{n-d}(F)}$ and $h_2: W_{n,d} \to B\underline{D_{1/d}^{\times}}$. In the special case $d=1$, $W_{n,1}$ is exactly the quotient $\mathcal{M}_{P}=\mathcal{M}_\infty / \underline{P_{n-1,1}}$ studied in \cite{johansson-ludwig}.

\begin{thm}Notation as above, there is a natural isomorphism
\[\mathcal{J}_n(\mathrm{Ind}_{P_{n-d,d}}^{\mathrm{GL}_n(F)}(\sigma_1 \boxtimes \sigma_2)) \cong R\Gamma(W_{n,d},h_{1, \et}^\ast \sigma_1 \otimes h_{2, \et}^\ast \mathcal{J}_d(\sigma_2)).
\]
\end{thm}
We like to think of this as a ``partial" K\"unneth formula. Although this theorem is logically independent of our other main results, the key point in the proof is still the new machinery provided by \cite{mann-p-adic-6-functors}.

Just as in \cite{johansson-ludwig}, this alternative formula leads to new vanishing results. Here we state the simplest such corollary.

\begin{cor}Let $\pi $ be an irreducible admissible representation of $\mathrm{GL}_3(F)$ which is not supersingular, and whose underlying vector space is infinite-dimensional. Then $\mathcal{J}^{4}(\pi)=0$.
\end{cor}

\subsection{A word on notation} \label{sec:intro.notation} This paper depends heavily on the mathematics developed in \cite{mann-p-adic-6-functors}. However, since DH is less derived than LM, there are some small notational differences between \cite{mann-p-adic-6-functors} and the present paper. In particular, we will write $R\mathrm{Hom}$ for the derived external Hom in any sheaf theory, and $R\Gamma$ for derived global sections. These are notated as $\mathrm{Hom}$ and $\Gamma$ in \cite{mann-p-adic-6-functors}. On the other hand, we follow \cite{mann-p-adic-6-functors} in writing $\IHom$ for the derived internal Hom, $f_{\ast}$ for derived pushforward, etc. We hope this does not cause any confusion.

\subsection{Acknowledgments} The idea that $\mathcal{J}$ should satisfy some form of duality was suggested to DH by Judith Ludwig during a visit to Heidelberg in January 2020. We heartily thank Judith for this key initial suggestion. We also sincerely thank Peter Scholze for some very helpful conversations about this material, and for suggesting that the authors might profitably collaborate. Some of these results were discovered during the Oberwolfach workshop ``Nonarchimedean geometry and applications" in February 2022, and we sincerely thank the organizers for the opportunity to participate in this event. We also thank Gabriel Dospinescu, Eugen Hellmann, Yongquan Hu, Christian Johansson, Vytautas Pa\v{s}k\={u}nas, Claus Sorensen, and Bogdan Zavyalov for some helpful conversations.

\section{Geometric preparations}

In this section we recall the $p$-adic 6-functor formalism constructed in \cite{mann-p-adic-6-functors} and extend it to a context where it can be applied to classifying stacks of $p$-adic Lie groups. In particular we extend the 6-functor formalism to incorporate shriek functors along ``stacky'' maps, similar to the stacky extension of the $\ell$-adic formalism provided in \cite{enhanced-ell-adic-6-functors}.

\subsection{Recollections on the solid six functors}

We recall the $p$-adic 6-functor formalism on small v-stacks defined in \cite{mann-p-adic-6-functors}. Let us first introduce the correct category of geometric objects to work with. In order to make sense of the sheaf $\ri^+_X/\pi$ on a small v-stack, we need to choose an ``untilt'' $X^\sharp$ of $X$ (this is automatically chosen if $X$ comes from an adic space over $\Z_p$) and a pseudouiformizer $\pi$ on $X^\sharp$. To make things more precise, we define an \emph{untilted small v-stack} $X^\sharp$ to be a small v-stack $X$ together with a chosen map $X \to \Spd\Z_p$. For every map $Y \to X$ from a perfectoid space $Y$ to $X$, the untilt $X^\sharp$ induces an untilt $Y^\sharp$ (in the sense of perfectoid spaces) of $Y$. The untilted small v-stack $X^\sharp$ comes equipped with a structure sheaf $\ri_{X^\sharp}$ and an integral structure sheaf $\ri^+_{X^\sharp} \subset \ri_{X^\sharp}$ on $X_\vsite$ such that on every affinoid perfectoid $Y \in X_\vsite$ we have $\ri_{X^\sharp}(Y) = \ri_{Y^\sharp}(Y^\sharp)$ (and similarly for $\ri^+_{X^\sharp}$). A \emph{pseudouniformizer} on $X^\sharp$ is a sheaf of ideals $\pi \subset \ri^+_{X^\sharp}$ such that for every affinoid perfectoid $Y = \Spa(B, B^+) \in X_\vsite$ the restriction of $\pi$ to $Y$ is a principal ideal and generated by a pseudouniformizer in $B^{\sharp+}$. We denote by $\vStackspi$ the category of pairs $(X^\sharp, \pi)$, where $X^\sharp$ is an untilted small v-stack and $\pi$ is a pseudouniformizer on $X^\sharp$. A map $(Y^\sharp, \pi_Y) \to (X^\sharp, \pi_X)$ in $\vStackspi$ is a map of small v-stacks $Y \to X$ over $\Spd\Z_p$ such that $\pi_X|_Y = \pi_Y$.

The first main result of \cite{mann-p-adic-6-functors} is the following:

\begin{thm}[{see \cite[Theorem 3.1.27]{mann-p-adic-6-functors}}] \label{rslt:def-of-DqcohriX}
There is a unique hypercomplete v-sheaf
\begin{align*}
    (\vStackspi)^\opp \to \infcatinf^\tensor, \qquad (X^\sharp, \pi) \mapsto \DqcohriX{X^\sharp}
\end{align*}
such that for every untilted affinoid perfectoid space $X^\sharp = \Spa(A^\sharp, A^{\sharp+})$ with pseudouniformizer $\pi \in A^{\sharp+}$, where $X$ is weakly of perfectly finite type over a totally disconnected space, we have
\begin{align*}
    \DqcohriX{X^\sharp} = \Dqcohri(A^{\sharp+}/\pi).
\end{align*}
\end{thm}

Here $\infcatinf^\tensor$ denotes the $\infty$-category of symmetric monoidal $\infty$-categories and $\Dqcohri(A^{\sharp+}/\pi)$ denotes the derived $\infty$-category of \emph{solid almost $A^{\sharp+}/\pi$-modules}. The term ``almost'' refers to the fact that we care about everything only up to $\mm_{A^\sharp}$-torsion. The term ``solid'' was introduced by Clausen-Scholze in \cite{condensed-mathematics} and roughly means ``complete topological''. See also \cite[\S2.2, \S2.9]{mann-p-adic-6-functors} for the precise definition of $\Dqcohri(A^{\sharp+}/\pi)$.

In the following we abstract a little further in order to simplify notation: We denote $\vStacksCoeff$ the category of pairs $(X, \Lambda)$, where $X$ is a small v-stack and $\Lambda$ is a system of integral torsion coefficients on $X$ (see \cite[Definition 3.2.10]{mann-p-adic-6-functors}). The main example for $\Lambda$ is $\Lambda = \ri^+_{X^\sharp}/\pi$ for an untilt $X^\sharp$ of $X$ and a pseudouniformizer $\pi$ on $X^\sharp$ -- the reader is invited to always assume this case in the following. To every $(X,\Lambda) \in \vStacksCoeff$ one can then define the $\infty$-category $\Dqcohri(X,\Lambda)$ in the same way as in \cref{rslt:def-of-DqcohriX}.

The second main result of \cite{mann-p-adic-6-functors} is the construction of a 6-functor formalism for the $\infty$-categories $\DqcohriX X$. In order to formulate it, let us recall the definition of an abstract 6-functor formalism from \cite[\S A.5]{mann-p-adic-6-functors}: Suppose we are given an $\infty$-category $\mathcal C$ of geometric objects and a collection of edges $E$ in $\mathcal C$ which contains all isomorphisms and is stable under composition and pullback. To the pair $(\mathcal C, E)$ one can then associate an $\infty$-operad $\Corr(\mathcal C)_{E,all}^\tensor$ which roughly has the following shape. The underlying $\infty$-category $\Corr(\mathcal C)_{E,all}$ has as objects the same objects as $\mathcal C$ and as morphisms the diagrams of the form $Y \from Y' \to X$ such that $[Y' \to X] \in E$, where the composition of morphisms is given by the obvious pullback diagrams. The operadic structure on $\Corr(\mathcal C)_{E,all}$ is roughly given by $X \tensor Y = X \times Y$ (if $\mathcal C$ admits products then this is literally true and $\Corr(\mathcal C)_{E,all}$ is a symmetric monoidal $\infty$-category). With this notation at hand, a \emph{pre-6-functor formalism} on $(\mathcal C, E)$ is defined to be a map of $\infty$-operads $\D\colon \Corr(\mathcal C)_{E,all}^\tensor \to \infcatinf^\times$, where $\infcatinf^\times$ denotes the cartesian structure on $\infcatinf$. Such a functor $\D$ encodes the following data: For every $X \in \mathcal C$ we get a symmetric monoidal $\infty$-category $\D(X)$, for every morphism $f\colon Y \to X$ in $\mathcal C$ we get a symmetric monoidal functor $f^*\colon \D(X) \to \D(Y)$ and for every morphism $[f\colon Y \to X] \in E$ we get a functor $f_!\colon \D(Y) \to \D(X)$ such that the functors $\tensor$, $f^*$ and $f_!$ satisfy the usual compatibilities. We say that $\D$ is a \emph{6-functor formalism} if the functors $\tensor$, $f^*$ and $f_!$ admit right adjoints.

In the setting of the desired $p$-adic 6-functor formalism, we take $\mathcal C = \vStacksCoeff$ and $E = bdcs$, the class of those maps $f\colon Y \to X$ of small v-stacks which are $p$-bounded (this roughly means that the map has finite $p$-cohomological dimension, see \cite[\S3.5]{mann-p-adic-6-functors}), locally compactifiable and representable in locally spatial diamonds. The main result of \cite{mann-p-adic-6-functors} now reads as follows:

\begin{thm}[{see \cite[Theorem 3.6.12]{mann-p-adic-6-functors}}] \label{rslt:cite-6-functor-formalism}
There is a 6-functor formalism
\begin{align*}
    \Dqcohri\colon \Corr(\vStacksCoeff)^\tensor_{bdcs,all} \to \infcatinf^\times, \qquad (X,\Lambda) \mapsto \Dqcohri(X,\Lambda)
\end{align*}
with the following properties:
\begin{thmenum}
    \item The symmetric monoidal $\infty$-categories $\Dqcohri(X,\Lambda)$ and pullback functors $f^*$ coincide with the version in \cref{rslt:def-of-DqcohriX}.
    \item For every bdcs morphism $f\colon Y \to X$ in $\vStacksCoeff$ the functor $f_!\colon \Dqcohri(Y,\Lambda) \to \Dqcohri(X,\Lambda)$ has the following properties: If $f$ is proper then $f_! = f_*$ is right adjoint to $f^*$. If $f$ is étale then $f_!$ is left adjoint to $f^*$.
\end{thmenum}
\end{thm}

\subsection{Six functors for stacky morphisms}

In order to efficiently apply the above 6-functor formalism to representation theory, i.e. sheaves on classifying stacks, it is very useful to generalize the definition of the shriek functors $f_!$ and $f^!$ to a larger class of maps $f$. Namely, we need to allow certain ``stacky'' maps $f$ as well, similar to what is done in \cite{enhanced-ell-adic-6-functors}.

\begin{defn}
We say that a bdcs map $f\colon Y \to X$ of small v-stacks admits \emph{universal $p$-codescent} if is satisfies the following property: Given any $(X', \Lambda) \in \vStacksCoeff$ which admits a map (in $\vStacksCoeff$) to some strictly totally disconnected space and given any map $X' \to X$ with base-change $f'\colon Y' \to X'$ and Čech nerve $Y'_\bullet \to X'$ the natural functor
\begin{align*}
    \Dqcohrishriek(X',\Lambda) \isoto \varprojlim_{n\in\Delta} \Dqcohrishriek(Y'_n,\Lambda)
\end{align*}
is an equivalence. Here $\Dqcohrishriek$ denotes the functor $(Z,\Lambda) \mapsto \Dqcohrishriek(Z,\Lambda)$, $f \mapsto f^!$.
\end{defn}

\begin{defn} \label{def:p-fine-maps}
A map $f\colon Y \to X$ of small v-stacks is called \emph{$p$-fine} if there is a map $g\colon Z \to Y$ such that $f \comp g$ is bdcs and $g$ is bdcs and admits universal $p$-codescent.
\end{defn}

\begin{lem}
\begin{lemenum}
    \item The condition of being $p$-fine is étale local on source and target.
    \item The collection of $p$-fine maps is stable under composition and base-change.
    \item Every bdcs map is $p$-fine. In particular every étale map and every map of rigid varieties over a fixed base field is $p$-fine.
	\item Let $f\colon Y \to X$ and $g\colon Z \to Y$ be maps of small v-stacks. If $f$ and $f \comp g$ are $p$-fine, then so is $g$.
\end{lemenum}
\end{lem}
\begin{proof}
It follows from \cite[Lemma 3.1.2.(4)]{liu-zheng-enhanced-operations} that bdcs maps admitting universal $p$-codescent are stable under composition. Part (ii) of the claim follows immediately and part (iii) is obvious (see also \cite[Lemma 3.6.10, Example 3.6.11]{mann-p-adic-6-functors}). Part (i) follows from the observation that étale maps admit universal $p$-codescent.

It remains to prove (iv), so let $f$ and $g$ be given as in the claim. Choose a bdcs map $Y' \to Y$ admitting universal $p$-codescent such that the composed map $Y' \to X$ is bdcs. Denote $g'\colon Z' \to Y'$ the base-change of $g$ along $Y' \to Y$. Since the map $Z' \to Z$ is bdcs, it follows from (ii) that the composed map $Z' \to X$ is $p$-fine. We can therefore choose a bdcs map $Z'' \to Z'$ admitting universal $p$-codescent such that $Z'' \to X$ is bdcs. Note that the map $Z'' \to X$ factors as $Z'' \to Y' \to X$ and since $Y' \to X$ is bdcs it follows from \cite[Lemma 3.6.10.(iv)]{mann-p-adic-6-functors} that $Z'' \to Y'$ is bdcs. In particular $Z'' \to Y$ is bdcs, so the map $Z'' \to Z$ shows that $g$ is $p$-fine.
\end{proof}

It is completely formal that the $p$-adic 6-functor formalism extends to all $p$-fine maps (in fact a similar extension can be performed in any 6-functor formalism):

\begin{prop} \label{rslt:stacky-6-functor-formalism}
The 6-functor formalism in \cref{rslt:cite-6-functor-formalism} extends uniquely to a 6-functor formalism
\begin{align*}
    \Dqcohri\colon \Corr(\vStacksCoeff)_{pfine,all} \to \infcatinf, \qquad (X,\Lambda) \mapsto \Dqcohri(X,\Lambda),
\end{align*}
where $pfine$ denotes the class of $p$-fine maps in $\vStacksCoeff$.
\end{prop}
\begin{proof}
Let $\mathcal C \subset \vStacksCoeff$ be the full subcategory spanned by those $(X,\Lambda)$ which admit a map (in $\vStacksCoeff$) to some strictly totally disconnected space. Then $\mathcal C$ is a basis of $\vStacksCoeff$, hence by \cite[Proposition A.5.16]{mann-p-adic-6-functors} and \cite[Proposition A.3.11.(i)]{mann-p-adic-6-functors} it is enough to construct the desired extension of the 6-functor formalism on $\mathcal C$, i.e. we need to construct a 6-functor formalism $\Dqcohri\colon \Corr(\mathcal C)_{pfine,all} \to \infcatinf$ extending the one from \cref{rslt:cite-6-functor-formalism}. We now apply \cite[Proposition A.5.14]{mann-p-adic-6-functors} with $E = bdcs$, $E' = pfine$ and $S \subset E$ being the subset of those maps which admit universal $p$-codescent. Condition (a) requires $\Dqcohri(X,\Lambda)$ to be presentable, which can be achieved by temporarily restricting to $\kappa$-condensed objects (as in the proof of \cite[Theorem 3.6.12]{mann-p-adic-6-functors}). Condition (b) follows from the definition of universal $p$-codescent, condition (c) follows from the definition of $p$-fine maps and condition (d) follows from \cite[Lemma 3.6.10]{mann-p-adic-6-functors} (see \cite[Remark A.5.15.(ii)]{mann-p-adic-6-functors}.
\end{proof}

\begin{rmk}
More concretely, the shriek functor $f_!\colon \Dqcohri(Y,\Lambda) \to \Dqcohri(X,\Lambda)$ along a $p$-fine map $f\colon Y \to X$ in $\vStacksCoeff$ is computed as follows. Via v-descent we can reduce to the case that $(X,\Lambda)$ admits a map to some strictly totally disconnected space. Now choose a bdcs map $g\colon Z \to Y$ such that $g$ admits universal $p$-codescent and $f \comp g$ is bdcs. Let $g_\bullet\colon Z_\bullet \to Y$ denote the Čech nerve of $g$. It follows from the $p$-codescent property of $g$ that for every $\mathcal M \in \Dqcohri(Y,\Lambda)$ the natural map $\varinjlim_{n\in\Delta} g_{n!} g_n^! \mathcal M \isoto \mathcal M$ is an isomorphism. Consequently we have
\begin{align*}
    f_! \mathcal M = f_! \varinjlim_{n\in\Delta} g_{n!} g_n^! \mathcal M = \varinjlim_{n\in\Delta} (f \comp g_n)_! g_n^! \mathcal M.
\end{align*}
Note that the maps $f \comp g_n\colon Z_n \to X$ are bdcs, hence $(f \comp g_n)_!$ was already defined in \cref{rslt:cite-6-functor-formalism}.
\end{rmk}

In order to make use of \cref{rslt:stacky-6-functor-formalism} we need a good amount of examples of bdcs maps admitting universal $p$-codescent. Arguably the most important class of such maps is formed by the $p$-cohomologically smooth maps (see \cite[Definition 3.8.1]{mann-p-adic-6-functors}). The following is an analog of \cite[Proposition 3.16]{enhanced-ell-adic-6-functors}:
\begin{lem}
Every bdcs and $p$-cohomologically smooth map of small v-stacks admits universal $p$-codescent.
\end{lem}
\begin{proof}
Let $f\colon Y \to X$ be a bdcs and $p$-cohomologically smooth map in $\vStacksCoeff$ with Čech nerve $f_\bullet\colon Y_\bullet \to X$; we need to show that the natural functor $f_\bullet^!\colon \Dqcohrishriek(X,\Lambda) \isoto \varprojlim_{n\in\Delta} \Dqcohrishriek(Y_n,\Lambda)$. This functor admits a left adjoint
\begin{align*}
    f_{\bullet!}\colon \varprojlim_{n\in\Delta} \Dqcohrishriek(Y_n,\Lambda) \to \Dqcohrishriek(X,\Lambda), \qquad \mathcal M_\bullet \mapsto \varinjlim_{n\in\Delta} f_{n!} \mathcal M_n.
\end{align*}
Note that $f_\bullet^!$ is conservative: This follows from the fact that $f_0^!$ is conservative, which in turn is implied by the fact that $f_0^!$ is just a twist of $f_0^*$ by definition of the $p$-cohomological smoothness of $f_0$. Thus in order to show that $f_\bullet^!$ is an equivalence, it is enough to show that $f_{\bullet!}$ is fully faithful, i.e. for given $\mathcal M_\bullet \in \varprojlim_{n\in\Delta} \Dqcohrishriek(Y_n,\Lambda)$ we need to see that the natural map $f_\bullet^! f_{\bullet!} \mathcal M_\bullet \isoto \mathcal M_\bullet$ is an isomorphism. This boils down to showing that the map $f_0^! f_{\bullet!} \mathcal M_\bullet \to \mathcal M_0$ is an isomorphism. Since $f_0$ is $p$-cohomologically smooth, $f_0^!$ preserves small colimits, so we end up with the claim that the natural map
\begin{align*}
    \varinjlim_{n\in\Delta} f_0^! f_{n!} \mathcal M_n \isoto \mathcal M_0
\end{align*}
is an isomorphism. Let us denote by $f'_\bullet\colon Y'_\bullet \to Y_0$ the base-change of $Y_\bullet \to X$ along $f_0\colon Y_0 \to X$. Denote $g_\bullet\colon Y'_\bullet \to Y_\bullet$ the natural projections. Then for every $n \ge 0$ the natural morphism $f'_{n!} g_n^! \mathcal M_n \isoto f_0^! f_{n!}\mathcal M_n$ is an isomorphism (by $p$-cohomological smoothness of $f_0$ and $g_n$ this reduces to proper base-change). On the other hand $g_\bullet^! \mathcal M_\bullet$ defines an object of $\varprojlim_{n\in\Delta} \Dqcohrishriek(Y'_n,\Lambda)$ so the above claimed isomorphism reduces to showing that the natural functor $\Dqcohrishriek(Y_0,\Lambda) \isoto \varprojlim_{n\in\Delta} \Dqcohrishriek(Y'_n,\Lambda)$ is an equivalence. This means that we can replace $Y_\bullet \to X$ by $Y'_\bullet \to Y_0$. But the latter augmented simplicial object is split, hence the desired equivalence of $\infty$-categories is formal.
\end{proof}

Let $\vStackspip \subset \vStackspi$ be the full subcategory spanned by those pairs $(X^\sharp, \pi)$ such that $\pi | p$. Recall the definition of the $\infty$-category $\DqcohriX{X^\sharp}^\varphi$ from \cite[Definition 3.9.10.(b)]{mann-p-adic-6-functors}: Roughly, this is the $\infty$-category of pairs $(\mathcal M, \varphi_{\mathcal M})$ where $\mathcal M \in \DqcohriX{X^\sharp}$ and $\varphi_{\mathcal M}$ is an isomorphism $\varphi_{\mathcal M}\colon \varphi^* \mathcal M \isoto \mathcal M$; here $\varphi$ denotes the Frobenius $x \mapsto x^p$ on $\ri^+_{X^\sharp}/\pi$. The above constructed 6-functor formalism extends to $\varphi$-modules:

\begin{prop}
The 6-functor formalism from \cref{rslt:stacky-6-functor-formalism} extends to a 6-functor formalism
\begin{align*}
    \Dqcohri(-)^\varphi\colon \Corr(\vStacksCoeff)_{pfine,all} \to \infcatinf, \qquad (X,\Lambda) \mapsto \Dqcohri(X,\Lambda)^\varphi.
\end{align*}
The forgetful functor $\Dqcohri(X,\Lambda)^\varphi \to \Dqcohri(X,\Lambda)$ preserves all small limits and colimits and commutes with the 6 functors $\tensor$, $\IHom$, $f^*$, $f_*$, $f_!$ and $f^!$. 
\end{prop}
\begin{proof}
By \cite[Proposition 3.9.13]{mann-p-adic-6-functors} the same is true for the 6-functor formalism on bdcs maps. To get this result also on $p$-fine maps, we note that the same extension procedure as in the proof of \cref{rslt:stacky-6-functor-formalism} applies verbatim to the $\varphi$-module case so that we get the desired 6-functor formalism. By \cite[Proposition 3.9.13]{mann-p-adic-6-functors} and functoriality of Kan extensions the forgetful functor preserves small limits and colimits and commutes with all the 6 functors except possibly with $f^!$. To see that it also commutes with $f^!$, note first that by the same argument as in the proof of \cite[Lemma 3.9.12.(v)]{mann-p-adic-6-functors} we can reduce to the case that the target of $f\colon Y \to X$ is a strictly totally disconnected space. Choose also a bdcs cover $g\colon Z \surjto Y$ such that $g$ admits universal $p$-codescent and $f \comp g$ is bdcs. Then by universal $p$-codescent the functor $g^!$ is conservative, hence to check that $f^!$ commutes with the forgetful functor it is enough to check that $g^! f^!$ commutes with the forgetful functor. This follows from the bdcs case.
\end{proof}

In the following we will often omit the explicit mention of untilt and pseudouniformizer and simply write $X \in \vStackspip$. The main reason for introducing $\varphi$-modules is that they have a tight connection to actual étale $\F_p$-sheaves: The $p$-torsion Riemann-Hilbert correspondence (see \cite[Theorem 3.9.23]{mann-p-adic-6-functors}) tells us that for $X \in \vStackspip$ there is a natural fully faithful embedding
\begin{align*}
    \D_\et(X, \F_p)^\oc \injto \DqcohriX X^\varphi, \qquad \mathcal F \mapsto \mathcal F \tensor_{\F_p} \ri^{+a}_X/\pi
\end{align*}
which induces an equivalence of dualizable objects on both sides. Here $\D_\et(X,\F_p)^\oc$ denotes the (derived) $\infty$-category of overconvergent $\F_p$-sheaves. The Riemann-Hilbert functor admits a right adjoint $(-)^\varphi\colon \DqcohriX X^\varphi \to \D_\et(X,\F_p)^\oc$.

\subsection{Cohomologically smooth maps} The notion of $p$-cohomologically smooth maps of small v-stacks generalizes in a straightforward way from bdcs maps to $p$-fine maps:

\begin{defn}
A $p$-fine map $f\colon Y \to X$ of small v-stacks is called \emph{$p$-cohomologically smooth} if for every map $X' \to X$ from a strictly totally disconnected perfectoid space $X' = \Spa(A', A'^+)$ there is a pseudouniformizer $\pi \in A'^+$ such that the map $f'\colon Y' := Y \times_X X' \to X'$ satisfies the following property: The natural morphism $f'^* \tensor f'^!(\ri^{+a}_{X'}/\pi) \isoto f'^!$ is an isomorphism of functors $\DqcohriX{X'} \to \DqcohriX{Y'}$ and $f'^!(\ri^{+a}_{X'}/\pi)$ is invertible.
\end{defn}

The basic results on bdcs and $p$-cohomologically smooth maps extend to the $p$-fine setting, as follows.

\begin{prop}
\begin{propenum}
    \item Let $f\colon Y \to X$ be a $p$-fine and $p$-cohomologically smooth map in $\vStacksCoeff$. Then the natural morphism
    \begin{align*}
        f^* \tensor f^!\Lambda^a \isoto f^!
    \end{align*}
    is an isomorphism of functors $\Dqcohri(X,\Lambda) \to \Dqcohri(Y,\Lambda)$ and $f^!\Lambda^a$ is invertible.
    
    \item Let
    \begin{center}\begin{tikzcd}
        Y' \arrow[r,"g'"] \arrow[d,"f'"] & X' \arrow[d,"f"]\\
        X' \arrow[r,"g"] & X
    \end{tikzcd}\end{center}
    be a cartesian square in $\vStacksCoeff$.
    \begin{enumerate}[(a)]
        \item Assume that $f$ is $p$-fine and that either $f$ or $g$ is $p$-fine and $p$-cohomologically smooth. Then the natural morphism
        \begin{align*}
            g'^* f^! \isoto f'^! g^*
        \end{align*}
        is an isomorphism of functors $\Dqcohri(X,\Lambda) \to \Dqcohri(Y',\Lambda)$.
        
        \item Assume that $g$ is $p$-fine and $p$-cohomologically smooth. Then the natural morphism
        \begin{align*}
            g^* f_* \isoto f'_* g'^*
        \end{align*}
        is an isomorphism of functors $\Dqcohri(Y,\Lambda) \to \Dqcohri(X',\Lambda)$.
    \end{enumerate}
    
    \item Let $f\colon Y \to X$ be a $p$-fine and $p$-cohomologically smooth map in $\vStacksCoeff$. Then for all $\mathcal M, \mathcal N \in \Dqcohri(X,\Lambda)$ the natural morphism
    \begin{align*}
        f^* \IHom(\mathcal M, \mathcal N) \isoto \IHom(f^* \mathcal M, f^* \mathcal N)
    \end{align*}
    is an isomorphism.
\end{propenum}
\end{prop}
\begin{proof}
We first note that part (i) is formal if $X$ is strictly totally disconnected and follows for general $X$ once we have shown (ii).(a) (see the proof of \cite[Proposition 3.8.4.(i)]{mann-p-adic-6-functors}). Moreover, once we have shown (ii).(a), also (iii) is formal (see the proof of \cite[Proposition 3.8.7]{mann-p-adic-6-functors}) and (ii).(b) is formal (see the proof of \cite[Proposition 3.8.6.(ii)]{mann-p-adic-6-functors}. Therefore, the whole claim boils down to showing (ii).(a). By the proof of \cite[Proposition 3.8.6.(iii)]{mann-p-adic-6-functors} the case that $g$ is $p$-cohomologically smooth follows formally from the case that $f$ is $p$-cohomologically smooth. Thus from now on we can assume that we are given a cartesian square as in (ii) with $f$ being $p$-fine and $p$-cohomologically smooth and we need to show that the morphism $g'^* f^! \isoto f'^! g^*$ is an isomorphism. As in the proof of \cite[Proposition 3.8.4.(ii)]{mann-p-adic-6-functors} we can formally reduce to the case that $X$ and $X'$ are strictly totally disconnected. Then by (i) (which holds unconditionally on (ii).(a) in the strictly totally disconnected case) we are reduced to showing that the natural map $\alpha\colon g'^* f^! \Lambda^a \isoto f'^! \Lambda^a$ is an isomorphism. We claim that $g'_*$ is conservative on bounded objects (in the sense of \cite[Definition 3.3.1.(a)]{mann-p-adic-6-functors}): Pick any cover $Z' \surjto X'$ such that the composition $Z \to X$ is bdcs (this can be done by the definition of $p$-fine maps). Then $Z'$ is a locally spatial diamond, so we can find a pro-étale cover of $Z'$ by totally disconnected spaces $Z_i$. Letting $W_i := Z_i \times_{X'} Y' = Z_i \times_X X'$ and denoting $g_i'\colon W_i \to Z_i$ the base-change of $g'$, we note that each $W_i$ is a $p$-bounded affinoid perfectoid space. Hence by \cite[Theorem 3.5.21]{mann-p-adic-6-functors} each $g'_i$ is a forgetful functor of module categories and in particular conservative. Using also bounded base-change (see \cite[Proposition 3.3.7]{mann-p-adic-6-functors}) we deduce that $g'$ is conservative on bounded objects. Now note that $f^! \Lambda^a$ is bounded: If $\Lambda = \ri^+_X/\pi$ for some $\pi | p$ then this follows from the Riemann-Hilbert correspondence (see \cite[Theorem 3.9.23]{mann-p-adic-6-functors}); for general $\Lambda$ one can make the usual reductions. Similarly $f'^!\Lambda^a$ is bounded. Thus altogether we deduce that it is enough to show that $g'_*\alpha\colon g'_* g'^* f^! \Lambda^a \isoto g'_* f'^! \Lambda^a$ is an isomorphism. The remaining proof of \cite[Proposition 3.8.4.(ii)]{mann-p-adic-6-functors} applies verbatim.
\end{proof}

\begin{lem}
\begin{lemenum}
    \item The condition of being $p$-fine and $p$-cohomologically smooth is analytically local on both source and target.
    
    \item Among $p$-fine maps, the condition of being $p$-cohomologically smooth is v-local on the target and $p$-cohomologically smooth local on the source.
    
    \item $p$-fine and $p$-cohomologically smooth maps are stable under composition and base-change.
\end{lemenum}
\end{lem}
\begin{proof}
This is formal, see \cite[Lemma 3.8.5]{mann-p-adic-6-functors}.
\end{proof}

\section{Representation-theoretic preparations}

In this section, we apply the 6-functor formalism to study representations of $p$-adic Lie groups. Since a number of categories will appear in what follows, we begin by introducing them. For the moment, we allow $G$ to be any locally profinite group with locally finite $p$-cohomological dimension.

First of all, we will need the $\infty$-category of solid $G$-representations $\D_\solid(G,\F_p)=\D_\solid(\F_p)^{BG}$, as defined in \cite[Definition 3.4.3]{mann-p-adic-6-functors}. This can be identified with the derived $\infty$-category of the category of continuous $G$-representations on solid $\F_p$-vector spaces (in the obvious sense). We also have the derived category $\D_\solid^{\sm}(G,\F_p)$ of smooth $G$-representations on solid $\F_p$-vector spaces \cite[Definition 3.4.11]{mann-p-adic-6-functors}. There is a natural functor $\D_\solid^{\sm}(G,\F_p) \to \D_\solid(G,\F_p)$, which is not fully faithful. Most classically, we have the usual derived category $\D^{\sm}(G,\F_p)$ of the abelian category of smooth $G$-representations on discrete $\F_p$-vector spaces. There is a natural identification $\D^{\sm}(G,\F_p) = \D_\solid^{\sm}(G,\F_p)_\omega$, identifying $\D^{\sm}(G,\F_p)$ as the full subcategory of $\D_\solid^{\sm}(G,\F_p)$ consisting of those smooth representations whose underlying $\F_p$-vector space is discrete (in the topological, i.e. condensed, sense), which we will use without comment. We also note that the composite functor $\D^{\sm}(G,\F_p) \to \D_{\solid}(G,\F_p)$ is fully faithful.

Now fix an algebraically closed nonarchimedean field $C$ of residue characteristic $p$, and a pseudouniformizer $\pi | p$. Let $B\underline{G} = \Spd(C)/ \underline{G}$ denote the classifying stack of $G$. We then have the $\infty$-category $\DqcohriX{B\underline{G}}^{\varphi}$, which (by \cite[Lemma 3.4.26.(ii)]{mann-p-adic-6-functors}) can be identified with the $\infty$-category of smooth $G$-representations on solid almost $\varphi$-modules over $\mathcal{O}_C^+/\pi$. This category can be directly accessed by the 6-functor formalism, as we will see. On the other hand, Proposition \ref{Detclassifying} gives a natural identification $\D^{\sm}(G,\F_p) = \D_{\et}(B\underline{G},\F_p)$. The Riemann-Hilbert functor then gives a natural fully faithful symmetric monoidal embedding
\[ \D^{\sm}(G,\F_p) \to \DqcohriX{B\underline{G}}^{\varphi}, \]
which factors through the full subcategory $\DqcohriX{B\underline{G}}^{\varphi}_{\omega}$ of discrete objects.

\subsection{Smooth representations}\label{smoothreps}
We continue to fix a locally profinite group $G$ of locally bounded $p$-cohomological dimension. 
In this section, we recall some results on the derived $\infty$-category $\D^{\sm}(G,\F_p)$. However, to place our results in context, we momentarily allow our coefficients to be any field $k$ (still assuming that $G$ has locally bounded $\mathrm{char} k$-cohomological dimension). 
Thus, let $\mathrm{Rep}_{k}^{\mathrm{sm}}(G)$
denote the category of smooth $G$-representations on $k$-vector
spaces, and let $\D^\sm(G,k)$ denote its derived $\infty$-category. When $k$ is clear from context, we just write $\D(G)=\D^\sm(G,k)$. We sometimes write $\mathbf{1}\in\mathrm{Rep}_{k}^{\mathrm{sm}}(G)$
for $k$ with the trivial $G$-action.

The assertions in the following proposition are well-known.
\begin{prop}
1. $\mathrm{Rep}_{k}^{\mathrm{sm}}(G)$ is a Grothendieck abelian
category. In particular, unbounded complexes have $K$-injective resolutions.

2. The derived category $\D^\sm(G,k)$ is left-complete and
has a natural symmetric monoidal structure given by $k$-linear tensor
product of cochain complexes.

3. If $H\subset G$ is an open subgroup, the restriction functor $r=r_{G}^{H}:\D^\sm(G,k)\to \D^\sm(H,k)$
is $t$-exact and symmetric monoidal.
\end{prop}

Let $\IHom_{G}(A,B)$ be the internal
hom in $\D^\sm(G,k),$ defined as a formal adjoint satisfying $\mathrm{Hom}_{\D(G)}(C\otimes_{k}A,B)=\mathrm{Hom}_{\D(G)}(C,\IHom_G(A,B))$. 
\begin{lem}\label{rhomforget}
If $H\subset G$ is an open subgroup, then there is a natural isomorphism
$r \IHom_G(A,B)\cong \IHom_H(rA,rB)$
for any $A,B\in \D(G)$. Here $r:\D(G)\to \D(H)$ denotes the restriction
functor.
\end{lem}

\begin{proof}
The functor $r$ has an left adjoint, given by the (exact) functor
$\mathrm{ind}_{H}^{G}(-)=-\otimes_{k[H]}k[G]$ of compact induction.
We then compute
\begin{align*}
\mathrm{Hom}_{\D(H)}(C,r\IHom_G(A,B)) & \cong\mathrm{Hom}_{\D(G)}(\mathrm{ind}_{H}^{G}C,\IHom_G(A,B))\\
 & \cong\mathrm{Hom}_{\D(G)}(A\otimes\mathrm{ind}_{H}^{G}C,B)\\
 & \cong\mathrm{Hom}_{\D(G)}(\mathrm{ind}_{H}^{G}(rA\otimes C),B)\\
 & \cong\mathrm{Hom}_{\D(H)}(rA\otimes C,rB)\\
 & \cong\mathrm{Hom}_{\D(H)}(C,\IHom_H(rA,rB))
\end{align*}
naturally in $C\in \D(H)$, so the result follows from Yoneda.
\end{proof}
If $A\in \D(G)$ is any object and $H\subset G$ is any open subgroup,
we write \[R\Gamma_{\sm}(H,A)=R\mathrm{Hom}_{\D(H)}(\mathbf{1},rA)=R\mathrm{Hom}_{\D(G)}(\mathrm{ind}_{H}^{G}\mathbf{1},A).\]
To actually compute this, let $A\overset{\sim}{\to}I^{\bullet}$ be
a $K$-injective resolution by injectives in $\mathrm{Rep}_{k}^{\mathrm{sm}}(G)$;
then $R\Gamma_{\sm}(H,A)\cong(I^{\bullet})^{H}$. When $H$ is a $p$-adic Lie group,
$R\Gamma_{\sm}(H,-)$ coincides with continuous group cohomology, by the main result in \cite{fust}.

\begin{lem}\label{colimcoh}
For any $A\in \D(G)$, we have $A\cong\mathrm{colim}_{H\subset G\,\mathrm{open}}R\Gamma_\sm(H,A)$. 
\end{lem}

\begin{proof}
If $A\overset{\sim}{\to}I^{\bullet}$ is a $K$-injective resolution
by injectives, then $R\Gamma_\sm(H,A)\cong(I^{\bullet})^{H}$ for any
$H$; here we use that restriction along $\mathrm{Rep}_{k}^{\mathrm{sm}}(G)\to\mathrm{Rep}_{k}^{\mathrm{sm}}(H)$
preserves $K$-injective complexes of injectives. But then
\begin{align*}
\mathrm{colim}_{H\subset G\,\mathrm{open}}R\Gamma_\sm(H,A) & \cong\mathrm{colim}_{H\subset G\,\mathrm{open}}(I^{\bullet})^{H}\cong I^{\bullet}\cong A,
\end{align*}
using the fact that each $I^{n}$ is a smooth representation.
\end{proof}
\begin{lem}\label{intexthom}
For any $A,B\in \D(G)$, we have $\IHom_G(A,B)\cong\mathrm{colim}_{H}R\mathrm{Hom}_{\D(H)}(A,B).$
\end{lem}
Here (and in what follows) we drop the restriction functors $r$ from the notation.

\begin{proof}
There are natural isomorphisms 
\begin{align*}
\mathrm{colim}_{H}R\mathrm{Hom}_{\D(H)}(A,B) & \cong\mathrm{colim}_{H}R\Gamma_\sm(H,\IHom_H(A,B))\\
 & \cong\mathrm{colim}_{H}R\Gamma_\sm(H,\IHom_G(A,B))\\
 & \cong \IHom_G(A,B),
\end{align*}
where the second line follows from Lemma \ref{rhomforget}, and the third line follows
from Lemma \ref{colimcoh}.
\end{proof}
\begin{prop}\label{ellnotpduality}
When $G$ is locally pro-$p$ and $\mathrm{char}k\neq p$, $\IHom_G(-,\mathbf{1})$
is the usual exact smooth duality functor, i.e. it sends any $\pi\in\mathrm{Rep}_{k}^{\mathrm{sm}}(G)$
to the smooth vectors in $\mathrm{Hom}_{k}(\pi,k)$.
\end{prop}

\begin{proof}
Easy, using the previous lemma together with the fact that $\mathbf{1}$
is an injective object of $\mathrm{Rep}_{k}^{\mathrm{sm}}(H)$ for
any pro-$p$ open compact $H\subset G$. (Use that any injective map
$V\to W$ in $\mathrm{Rep}_{k}^{\mathrm{sm}}(H)$ admits an $H$-equivariant
retraction, by averaging any $k$-linear retraction against a $k$-valued
Haar measure.)
\end{proof}
When $G$ is a $p$-adic Lie group and $k = \F_p$, the naive smooth duality functor on $\mathrm{Rep}_{\F_p}^{\mathrm{sm}}(G)$
is rather useless, because the functor of passage to smooth vectors
is no longer exact. To remedy this, Kohlhaase has introduced a family
of derived smooth duality functors $\mathcal{S}_{G}^{i}(-)$ on $\mathrm{Rep}_{\F_p}^{\mathrm{sm}}(G)$,
which are essentially the right-derived functors of the functor $V\mapsto(V^{\ast})^{\mathrm{sm}}$.
More precisely, Kohlhaase defines
\[
\mathcal{S}_{G}^{i}(V)=\mathrm{colim}_{H\subset G\,\mathrm{open}\,\mathrm{compact}}\mathrm{Ext}_{\F_p[[H]]}^{i}(k,V^{\ast}).
\]
Here $V^{\ast}=\mathrm{Hom}_{\F_p}(V,\F_p)$ denotes the Pontryagin dual
of $V$ with its natural pseudocompact $\F_p$-module structure; this
is naturally a pseudocompact $\F_p[[H]]$-module for any open compact
$H\subset G$. We prefer to package the $\mathcal{S}^{i}$'s into
a single functor on the derived category. In other words, we consider
the endofunctor of $\D^\sm(G,\F_p)$ defined by
\[
V\mapsto\mathcal{S}_{G}(V)=\mathrm{colim}_{H\subset G\,\mathrm{open}\,\mathrm{compact}}R\mathrm{Hom}_{\F_p[[H]]}(\F_p,V^{\ast}),
\]
where $V^{\ast}\in \D(\F_p[[H]]-\mathrm{mod})$ is computed termwise.
Here $R\mathrm{Hom}_{\F_p[[H]]}$ denotes $R\mathrm{Hom}$ in the derived
category of $\F_p[[H]]$-modules. Of course, $H^{i}(\mathcal{S}(V))=\mathcal{S}^{i}(V)$ when $V$ is concentrated in degree zero. 

The following proposition was noted independently by Schneider-Sorensen \cite{schneider-sorensen-dual-reps}.

\begin{prop}
As endofunctors of $\D^\sm(G,\F_p)$, the functor $\IHom_G(-,\mathbf{1})$
coincides with Kohlhaase's derived smooth duality functor $\mathcal{S}_{G}(-)$.
\end{prop}

\begin{proof}
As noted above, $\mathcal{S}_{G}(V)$ is defined by the formula
\[
\mathcal{S}_{G}(V)=\mathrm{colim}_{H\subset G\,\mathrm{open}\,\mathrm{compact}}R\mathrm{Hom}_{\F_p[[H]]}(\F_p,V^{\ast}).
\]
To actually compute this, choose a uniformly powerful open pro-$p$ subgroup
$G_{0}\subset G$. By work of Lazard \cite{lazard-monster}, we can choose a finite length
resolution $P_{\bullet}\to \F_p\to0$ of $\F_p$ by finite free $\F_p[[G_{0}]]$-modules.
Then we can express $\mathcal{S}_{G}(-)$ more concretely as
\[
\mathcal{S}_{G}(V)=\mathrm{colim}_{H\subset G_{0}\,\mathrm{open}}\mathrm{Hom}_{\F_p[[H]]}(P_{\bullet},V^{\ast}).
\]

Set $I^{n}=\mathrm{Hom}_{\F_p}^{\mathrm{cts}}(P_{n},\F_p)$, so $0\to \F_p \to I^{\bullet}$
is a finite resolution of $\mathbf{1}$ by injective objects in $\mathrm{Rep}_{\F_p}^{\mathrm{sm}}(G_{0})$.
Since Pontryagin duality gives an exact anti-equivalence between smooth
$H$-representations over $\F_p$ and pseudocompact $\F_p[[H]]$-modules, we have
natural identifications
\[
\mathrm{Hom}_{\F_p[[H]]}(P_{\bullet},V^{\ast})\cong\mathrm{Hom}_{\mathrm{Rep}_{\F_p}^{\mathrm{sm}}(H)}(V,I^{\bullet})\cong R\mathrm{Hom}_{\D(H)}(V,\mathbf{1}).
\]
Putting things together, we get 
\begin{align*}
\mathcal{S}_{G}(V) & \cong\mathrm{colim}_{H\subset G_{0}\,\mathrm{open}}\mathrm{Hom}_{\mathrm{Rep}_{\F_p}^{\mathrm{sm}}(H)}(V,I^{\bullet})\\
 & \cong\mathrm{colim}_{H\subset G_{0}\,\mathrm{open}}R\mathrm{Hom}_{\D(H)}(V,\mathbf{1})\\
 & \cong \IHom_G(V,\mathbf{1}),
\end{align*}
where the last line follows by Lemma \ref{intexthom}.
\end{proof}
In the remainder of this subsection, we continue to assume that $G$ is a $p$-adic Lie group.
Recall that the \emph{dimension }of an admissible smooth representation
$V\in\mathrm{Rep}_{\F_p}^{\mathrm{sm}}(G)$ is the quantity
\[
\mathrm{dim}_{G}V=\mathrm{dim}\,G\cdot\lim_{n\to\infty}\frac{\log\dim_{\F_p}V^{K_{n}}}{\log\,[K_{n}:K_{0}]}=\lim_{n\to\infty}\frac{\log\dim_{\F_p}V^{K_{n}}}{n\log p}.
\]
Here $K_{0}\subset G$ is a fixed choice of a uniformly powerful open
pro-$p$ subgroup, and $K_{n}=K^{p^{n}}\subset K_{0}$ is the closed
normal subgroup generated by the $p^{n}$th powers of elements of
$K$. (One automatically has that $[K_{n}:K_{0}]=p^{n\cdot\dim G}$,
which explains the second equality above.) By \cite[Proposition 2.18]{emerton-paskunas}, this limit exists independently
of the choice of $K_{0}$, and defines an
integer in the interval $[0,\dim G]$. Moreover, $\mathrm{dim}_{G}V=0$
iff $\mathrm{dim}_{\F_p}V<\infty$. Additionally, if $0\to V'\to V\to V''\to0$
is an exact sequence of admissible smooth representations, then $\mathrm{dim}_{G}V=\max(\mathrm{dim}_{G}V',\dim_{G}V'')$.
In particular, any subquotient $W$ of a given $V$ satisfies $\mathrm{dim}_{G}W\leq\mathrm{dim}_{G}V$.

A key point for us is that $\dim_{G}$ is closely related to the behavior
of the functors $\mathcal{S}_{G}^{i}$.
\begin{prop}\label{dimbasics}
\begin{propenum}
    \item If $V$ is an admissible smooth $G$-representation, then each $\mathcal{S}_{G}^{i}(V)$ is admissible smooth, with $\mathrm{dim}_{G}\mathcal{S}_{G}^{i}(V)\leq i$ for all $i$. Moreover, if $d=\mathrm{dim}_{G}V$ then $\mathrm{dim}_{G}\mathcal{S}_{G}^{d}(V)=d$ and $\mathcal{S}_{G}^{i}(V)=0$ for all $i>d$.

    \item For any $A\in \D^\sm(G,\F_p)$ whose cohomologies $H^i(A)$ are all admissible, the complex $\mathcal{S}_G(A)$ has admissible cohomologies, and the biduality map $A\to(\mathcal{S}_{G}\circ\mathcal{S}_{G})(A)$ is an isomorphism.

    \item If $V$ is an admissible smooth representation with $\dim_{G}V=0$, so $\dim_{\F_p}V<\infty$, then $\mathcal{S}(V)\cong\mathcal{S}^{0}(V)$ is the naive linear dual of $V$.
\end{propenum}
\end{prop}

\begin{proof}
The first part is essentially a summary of results from \cite{kohlhaase}. The second part is Pontryagin-dual to the biduality discussed
in the proof of \cite[Theorem 3.5]{kohlhaase}; see also \cite[Proposition 3.4]{schneider-sorensen-dual-reps}. The third part is immediate from the first part.
\end{proof}

Note that part (i) here implies that $\dim_{G}V=\sup\left\{ i\mid\mathcal{S}_{G}^{i}(V)\neq0\right\} $. We also observe that $\mathcal{S}^{0}(-)$ can be completely understood.

\begin{prop}\label{dimzero}
Let $V$ be any admissible smooth $G$-representation. Then $\mathcal{S}_{G}^{0}(V)\neq0$
if and only if $V$ admits a nonzero $G$-stable quotient $V\twoheadrightarrow W$
with $\mathrm{dim}_{\F_p}W<\infty$.
\end{prop}

\begin{proof}
``If'': Given such a $V\to W$, we get a nonzero map $\mathcal{S}(W)\to\mathcal{S}(V)$.
But the previous proposition shows that $\mathcal{S}(W)=\tau^{\leq0}\mathcal{S}(W)=\mathcal{S}^{0}(W)$,
so we get a nonzero map $\mathcal{S}^{0}(W)\to\mathcal{S}(V)$ which
necessarily factors over a nonzero map $\mathcal{S}^{0}(W)\to\mathcal{S}^{0}(V)$.

``Only if'': Suppose $\mathcal{S}^{0}(V)\neq0$. By the previous
proposition, this is zero-dimensional as a $G$-representation, so
$\mathcal{S}(\mathcal{S}^{0}(V))\cong\mathcal{S}^{0}(\mathcal{S}^{0}(V))$.
In particular, applying $\mathcal{S}$ to the map $\mathcal{S}^{0}(V)=\tau^{\leq0}\mathcal{S}(V)\to\mathcal{S}(V)$
and using biduality, we get a nonzero map $V\cong\mathcal{S}(\mathcal{S}(V))\to\mathcal{S}(\mathcal{S}^{0}(V))\cong\mathcal{S}^{0}(\mathcal{S}^{0}(V))$,
and $\mathcal{S}^{0}(\mathcal{S}^{0}(V))$ is finite-dimensional over
$\F_p$.
\end{proof}

We conclude this section by giving a geometric interpretation of $\D^\sm(G,\F_p)$. Let $C$ be any algebraically closed nonarchimedean field of residue characteristic $p$.

\begin{prop}\label{Detclassifying} Notation as above, there are natural symmetric monoidal equivalences
\[ \D_{\et}([\Spd(C)/G],\F_p) \cong \D([\Spd(C)/G]_{\et},\F_p) \cong \D^\sm(G,\F_p).\]
In particular, the first two categories are left-complete and canonically independent of $C$.
\end{prop}
\begin{proof}This follows from the proof of Theorem V.1.1 in \cite{FS}, noting that that $[\Spd(C)/G]_{\et}$ locally has uniformly bounded $p$-cohomological dimension by Lazard-Serre.
\end{proof}

\subsection{Sheaves on classifying stacks}

We now study the stacky 6-functor formalism introduced above in the context of the classifying stack $*/G$ associated to a locally profinite group $G$. Recall that if $G$ has locally finite $p$-cohomological dimension then for every totally disconnected space $X = \Spa(A, A^+)$ with pseudouniformizer $\pi$ the $\infty$-category $\DqcohriX{X/G}$ can be identified with the $\infty$-category of smooth $G$-representations on solid almost $A^+/\pi$-modules (see \cite[Lemma 3.4.26.(ii)]{mann-p-adic-6-functors}). In the following we will show that if $G$ has locally finite $p$-cohomological dimension then $*/G \to *$ is $p$-fine, and we will deduce a precise criterion for when the map $*/G \to *$ is $p$-cohomologically smooth. In particular, this will hold in all practical situations, see Theorem \ref{padicLiePoincare} below.

To start things off, we first translate the finite $p$-cohomological dimension condition into a vanishing of Ext-groups.  Also recall the definition of the $p$-cohomological dimension $\cd_p G$ from \cite[Definition 3.4.20]{mann-p-adic-6-functors} (see \cite[Proposition 3.4.22]{mann-p-adic-6-functors} to relate this to more classical versions).

\begin{prop} \label{rslt:Ext-i-vanishing-for-discrete-G-rep}
Let $G$ be a profinite group and let $M, N \in \D_\solid(\F_p)^{BG}$ be static continuous $G$-representations on solid $\F_p$-vector spaces. If $M$ is discrete then
\begin{align*}
    \Ext^i_G(M, N) = 0 \qquad \text{for $i > \cd_p G$}.
\end{align*}
\end{prop}
\begin{proof}
The $\infty$-category $\D_\solid(\F_p)^{BG}$ comes naturally equipped with a symmetric monoidal structure such that the forgetful functor $q^*\colon \D_\solid(\F_p)^{BG} \to \D_\solid(\F_p)$ is symmetric monoidal (e.g. one can construct $\D_\solid(\F_p)^{BG}$ as certain sheaves on the classifying stack $*/G$ in a suitable geometric setup). The forgetful functor $q^*$ has a left adjoint $q_\natural\colon \D_\solid(\F_p) \to \D_\solid(\F_p)^{BG}$ which is given by $q_\natural(M) = \F_{p\solid}[G] \tensor_{\F_{p\solid}} M$. It is straightforward to check that $q_\natural$ satisfies the projection formula, i.e. that for all $M \in \D_\solid(\F_p)^{BG}$ and $N \in \D_\solid(\F_p)$ the natural morphism $q_\natural(N \tensor q^* M) \isoto q_\natural N \tensor M$ is an isomorphism. The symmetric monoidal structure on $\D_\solid(\F_p)^{BG}$ is closed, i.e. admits internal hom's $\IHom$ (this is clear if we restrict to $\kappa$-condensed objects for some cardinal $\kappa$ and then follows for unrestricted condensed objects by the following computation). It follows formally from the projection formula for $q_\natural$ that for all $M, N \in \D_\solid(\F_p)^{BG}$ the natural morphism
\begin{align*}
    q^* \IHom(M, N) \isoto \IHom(q^* M, q^* N)
\end{align*}
is an isomorphism -- in other words, the internal hom in $\D_\solid(\F_p)^{BG}$ is computed on the underlying solid $\F_p$-vector spaces. In particular, if $M$ and $N$ are static and $M$ is discrete (and thus $q^* M \isom \bigdsum_I \F_p$ for some set $I$) then $\IHom(M, N)$ is static. Therefore, from
\begin{align*}
    \Ext^i_G(M, N) = \pi_0 R\Hom(M, N[i]) = \pi_0 \Hom(\F_p, \IHom(M, N[i])) = H^i(G, \IHom(M, N))
\end{align*}
we immediately deduce the claim.
\end{proof}

We will now show that $*/G \to *$ is $p$-fine. This reduces to showing that the canonical cover $* \to */G$ admits universal $p$-codescent, which will be established by adapting Mathew's ideas on descendable algebras \cite{akhil-galois-group-of-stable-homotopy}.

\begin{lem} \label{rslt:p-codescent-for-classifying-stack}
Let $G$ be a profinite group with $\cd_p G < \infty$. Then the map $* \to */G$ of small v-stacks is bdcs and admits universal $p$-codescent.
\end{lem}
\begin{proof}
It is clear that the map $* \to */G$ is bdcs (e.g. the pullback along itself is $G \to *$, which is pro-étale). Now pick any $(X,\Lambda) \in \vStacksCoeff$ admitting a map to a strictly totally disconnected space, pick any map $X \to */G$ and denote $q\colon Y \to X$ the base-change of $* \to */G$. We need to show that $q$ admits $p$-codescent, i.e. if $q_\bullet\colon Y_\bullet \to X$ denotes the Čech nerve of $q$ then the natural functor
\begin{align*}
    \Dqcohrishriek(X, \Lambda) \isoto \varprojlim_{n\in\Delta} \Dqcohrishriek(Y_n, \Lambda)
\end{align*}
is an equivalence. To see this we employ Lurie's Beck-Chevalley condition, see \cite[Corollary 4.7.5.3]{lurie-higher-algebra}. Part (2) of this condition is easily satisfied: Since all the maps in the diagram $Y_\bullet \to X$ are proper, the required left adjointability for the upper shriek functors reduces to proper base-change (by passing to left adjoints). Thus in order to apply Lurie's result it remains to verify the following:
\begin{enumerate}[(a)]
    \item The functor $q^!\colon \Dqcohri(X,\Lambda) \to \Dqcohri(Y,\Lambda)$ is conservative.
    \item The functor $q^!$ preserves geometric realizations of $q^!$-split simplicial objects in $\Dqcohri(X,\Lambda)$.
\end{enumerate}
To prove (a) and (b) we apply ideas by Mathew \cite{akhil-galois-group-of-stable-homotopy} (which were also used extensively in the setup of the $p$-adic 6-functor formalism, cf. \cite[\S2.6]{mann-p-adic-6-functors}). More concretely, we will prove the following claim:
\begin{itemize}
    \item[($\ast$)] Let $\langle q_*q^! \rangle \subset \Fun(\Dqcohri(X,\Lambda), \Dqcohri(X,\Lambda))$ be the full subcategory generated by the endofunctor $q_*q^!$ under compositions, retracts and finite (co)limits. Then $\langle q_*q^! \rangle$ contains the identity functor.
\end{itemize}
Let us first see how ($\ast$) implies (a) and (b), so assume that ($\ast$) holds for the moment. To prove (a), let $\mathcal M \to \mathcal N$ be a given map in $\Dqcohri(X,\Lambda)$ such that $q^! \mathcal M \isoto q^! \mathcal N$ is an isomorphism. Then also $q_* q^! \mathcal M \isoto q_* q^! \mathcal N$ is an isomorphism and one deduces easily that $F(\mathcal M) \isoto F(\mathcal N)$ is an isomorphism for all $F \in \langle q_*q^! \rangle$; taking $F = \id$ shows that $\mathcal M \isoto \mathcal N$ is an isomorphism, as desired. To prove (b), let $\mathcal M_\bullet$ be a $q^!$-split simplicial object in $\Dqcohri(X,\Lambda)$. Then $q_*q^! \mathcal M_\bullet$ is a constant $\Ind$-object (i.e. lies in the essential image of $\Dqcohri(X,\Lambda) \injto \Ind(\Dqcohri(X,\Lambda))$) and it follows immediately that the same is true for $F(\mathcal M_\bullet)$ for all $F \in \langle q_*q^! \rangle$. Taking $F = \id$ shows that $\mathcal M_\bullet$ is a constant $\Ind$-object, so that its colimit commutes with any exact functor; in particular it commutes with $q^!$.

It now remains to verify ($\ast$). We compute (using \cite[Corollary 3.6.17.(i)]{mann-p-adic-6-functors})
\begin{align*}
    q_*q^! = q_* \IHom(\Lambda^a, q^!) = \IHom(q_! \Lambda^a, -) = \IHom(q_* \Lambda^a, -).
\end{align*}
Therefore ($\ast$) can be reduced to showing the following: Let $\langle q_*\Lambda^a \rangle \subset \Dqcohri(X,\Lambda)$ be the full subcategory generated by $q_*\Lambda^a$ under tensor product, retracts and finite (co)limits; then $\langle q_*\Lambda^a \rangle$ contains the tensor unit $\Lambda^a$. By \cite[Proposition 3.20]{akhil-galois-group-of-stable-homotopy} this is equivalent to the following claim:
\begin{itemize}
    \item[($\ast'$)] The algebra object $q_*\Lambda^a \in \Dqcohri(X,\Lambda)$ admits descent.
\end{itemize}
It remains to prove that ($\ast'$) is true. We first perform some reductions. By assumption on $(X,\Lambda)$ there is a map $(X,\Lambda) \to (Z,\Lambda)$ with $Z = \Spa(A, A^+)$ being strictly totally disconnected. We thus get an induced map $X \to Z \times */G = Z/G$ such that $q$ is the pullback of the projection $q_Z\colon Z \to Z/G$. The induced pullback functor $\Dqcohri(Z/G,\Lambda) \to \Dqcohri(X,\Lambda)$ is exact, symmetric monoidal and sends $q_{Z*}\Lambda^a$ to $q_*\Lambda^a$ (by proper base-change). It is therefore enough to prove that ($\ast$') holds for $q_Z$ in place of $q$, so from now on we assume $X = Z/G$ and hence $Y = Z$.

Given a morphism $\Lambda' \to \Lambda$ of integral torsion coefficients on $Z$, the induced base-change functor $\Dqcohri(Z,\Lambda') \to \Dqcohri(Z,\Lambda)$ preserves tensor products, retracts and finite (co)limits and sends $q_*\Lambda'^a$ to $q_*\Lambda^a$ (by \cite[Proposition 3.6.16]{mann-p-adic-6-functors}), hence if ($\ast'$) holds for $\Lambda'$ then it also holds for $\Lambda$. By definition of integral torsion coefficients we can thus reduce to the case that $\Lambda = \ri^{+a}_{Z^\sharp}/\pi$ for some untilt $Z^\sharp$ of $Z$ and some pseudouniformizer $\pi$ on $Z^\sharp$. Pick some pseudouniformizer $\varpi$ on $Z^\sharp$ such that $\varpi | p$ and some integer $n > 0$ such that $\pi | \varpi^n$. By \cite[Proposition 3.35]{akhil-galois-group-of-stable-homotopy} the map of discrete rings $A^{\sharp+}/\pi \to A^{\sharp+}/\varpi$ admits descent, which after pullback along $Z/G \to Z$ implies that the map $\ri_{Z^\sharp/G}^{+a}/\pi \to \ri_{Z^\sharp/G}^{+a}/\varpi$ in $\DqcohriX{Z^\sharp/G}$ admits descent. Therefore, if we can show that ($\ast'$) holds for $\Lambda = \ri^+_{Z^\sharp/G}/\varpi$ then by \cite[Proposition 3.24]{akhil-galois-group-of-stable-homotopy} and \cite[Lemma 11.10]{bhatt-scholze-witt} it also follows for $\Lambda = \ri^+_{Z^\sharp/G}/\pi$. Thus from now on we can assume that $\pi = \varpi$, i.e. we now have $\pi | p$.

Now recall the Riemann-Hilbert functor $RH\colon \D_\et(Z/G,\F_p)^\oc \to \DqcohriX{Z^\sharp/G}$ from \cite[Definition 3.9.21.(a)]{mann-p-adic-6-functors} (where we ignore the $\varphi$-module structure). Then $RH$ is symmetric monoidal and exact and satisfies $RH(q_*\F_p) = q_*(\ri_{Z^\sharp}^{+a}/\pi)$. To verify the latter, note that this identity can be checked after applying $q^*$, where we get (using base-change)
\begin{align*}
    &q^* RH(q_* \F_p) = RH(q^*q_* \F_p) = RH(\cts(G,\F_p)) = \cts(G, \F_p) \tensor_{\F_p} A^{\sharp+,a}/\pi = \cts(G, A^{\sharp+,a}/\pi) \\&\qquad= q^* q_*(A^{\sharp+,a}/\pi)
\end{align*}
(here we implicitly use the identity $\DqcohriX{Z^\sharp} = \Dqcohri(A^{\sharp+}/\pi)$). All in all this means that it is enough to show that $q_*\F_p \in \D_\et(Z/G, \F_p)^\oc$ is descendable. It follows easily from \cite[Lemma 3.9.18.(iii)]{mann-p-adic-6-functors} that $\D_\et(Z/G, \F_p)^\oc = \D^\sm(G, \F_p(Z))_\omega$ is the derived $\infty$-category of discrete smooth $G$-representations on $\F_p(Z)$-modules (use the same argument as in the proof of \cite[Lemma 3.4.26.(iii)]{mann-p-adic-6-functors}). Our goal is to show that the $G$-representation $\cts(G, \F_p(Z))$ admits descent in this $\infty$-category. By considering base-change along $\F_p \to \F_p(Z)$ we reduce the claim to the following:
\begin{itemize}
    \item[($\ast''$)] The algebra object $\cts(G,\F_p) \in \D^\sm(G, \F_p)_\omega$ admits descent.
\end{itemize}
Note that $\D^\sm(G,\F_p)_\omega$ is equivalent to the full subcategory of $\D_\solid(\F_p)^{BG}$ spanned by the discrete $G$-representations (e.g. both $\infty$-categories can be constructed using the same descent diagram), hence by \cref{rslt:Ext-i-vanishing-for-discrete-G-rep} it follows that for all static $M, N \in \D^\sm(G,\F_p)_\omega$ we have $\pi_0\Hom(M, N[i]) = 0$ for $i > d := \cd_p G$. Now consider the triangle
\begin{align*}
    \F_p \to \Tot_d(\cts(G,\F_p)^{\tensor \bullet+1}) \to X
\end{align*}
in $\D^\sm(G,\F_p)_\omega$, where $\Tot_d$ denotes the $d$-truncated totalization of a cosimplicial object (i.e. the limit over $\Delta_{\le d}$). We know that $\Tot(\cts(G,\F_p)^{\tensor \bullet+1}) = \F_p$ (by descent) and by the dual version of \cite[Proposition 1.2.4.5]{lurie-higher-algebra} we have $\pi_i \Tot_d = \pi_i \Tot$ for $i > -d$. It follows that $X$ is concentrated in homological degrees $\le -d$, so that by the above vanishing of hom's the connecting map $X \to \F_p[1]$ must be zero. It follows that $\Tot_d(\cts(G,\F_p)^{\tensor \bullet+1}) \isom \F_p \dsum X$, which immediately implies that $\cts(G,\F_p)$ admits descent.
\end{proof}

\begin{cor} \label{rslt:finite-cd-p-implies-classifying-stack-is-pfine}
Let $G$ be a profinite group with $\cd_p G < \infty$. Then:
\begin{corenum}
    \item The map $*/G \to *$ is $p$-fine.
    \item For every $(X,\Lambda) \in \vStacksCoeff$ with associated map $f\colon X/G \to X$ there is a natural isomorphism $f_! = f_*$ of functors $\Dqcohri(X/G,\Lambda) \to \Dqcohri(X,\Lambda)$.
\end{corenum}
\end{cor}
\begin{proof}
Part (i) follows from \cref{rslt:p-codescent-for-classifying-stack} by using the factorization $\id = [* \to */G \to *]$. To prove (ii) let $(X,\Lambda)$ be given and assume first that $(X,\Lambda)$ admits a map to some strictly totally disconnected space. Let $q\colon X \to X/G$ denote the projection and let $q_\bullet\colon X_\bullet \to X/G$ be the Čech nerve of $q$. By \cref{rslt:p-codescent-for-classifying-stack} we have $\Dqcohrishriek(X/G,\Lambda) = \varprojlim_{n\in\Delta} \Dqcohrishriek(X_n,\Lambda)$, which means that for all $\mathcal M \in \Dqcohri(X/G,\Lambda)$ the natural map $\varinjlim_{n\in\Delta} q_{n!} q_n^! \mathcal M \isoto \mathcal M$ is an isomorphism. Since both $f_*$ and $f_!$ preserve small colimits (for $f_*$ this follows from $\cd_p G < \infty$) and all $q_n$ and $f \comp q_n$ are proper we deduce
\begin{align*}
    f_* = f_* \varinjlim_n q_{n!} q_n^! = \varinjlim_n (f \comp q_n)_* q_n^! = \varinjlim_n (f \comp q_n)_! q_n^! = f_!.
\end{align*}
The case of general $(X,\Lambda)$ can be deduced via descent (more precisely, do the above computation on sections of the cocartesian fibrations over $\Delta$).
\end{proof}

Having established that $*/G \to *$ is $p$-fine in many cases, we now determine criteria for when this map is $p$-cohomologically smooth. When $G$ has finite $p$-cohomological dimension, the precise answer is captured by the following definition.

\begin{defn} \label{def:profinite-Poincare-group}
Let $G$ be a profinite group with $\cd_p G < \infty$. We say that $G$ is \emph{$p$-Poincaré of dimension $d$} if it satisfies the following conditions:
\begin{enumerate}[(i)]
    \item $H^i(G, M)$ is finite for all $i \ge 0$ and all $G$-representations $M$ on finite $\F_p$-vector spaces.
    \item The solid $\F_p$-vector space $R\Gamma(G, \F_{p\solid}[G])$ is invertible and concentrated in cohomological degree $d$ (i.e. it is isomorphic to $\F_p[-d]$). 
\end{enumerate}
\end{defn}

Note that condition (i) is equivalent to the finiteness of $H^i(H,\F_p)$ for all $i\geq 0$ and all open subgroups $H\subseteq G$. Note also that if $G$ is $p$-Poincaré of dimension $d$, then so is every open subgroup of $G$.

\begin{prop} \label{rslt:p-Poincare-equiv-p-cohom-smoothness}
Let $G$ be a profinite group with $\cd_p G < \infty$.
\begin{propenum}
    \item $G$ satisfies condition (i) of \cref{def:profinite-Poincare-group} if and only if for all strictly totally disconnected spaces $(X,\Lambda) \in \vStacksCoeff$ with associated map $f\colon X/G \to X$ the morphism $f^* \tensor f^! \Lambda^a \isoto f^!$ is an equivalence of functors $\Dqcohri(X,\Lambda) \to \Dqcohri(X/G,\Lambda)$.
    \item $G$ is a $p$-Poincaré group (of dimension $d$) if and only if the map $*/G \to *$ is $p$-cohomologically smooth (of pure dimension $d/2$).
\end{propenum}
\end{prop}
\begin{proof}
Fix a strictly totally disconnected space $(X,\Lambda) \in \vStacksCoeff$ with associated map $f\colon X/G \to X$. Then $\Dqcohri(X,\Lambda) = \Dqcohri(\Lambda(X))$ is the $\infty$-category of solid almost $\Lambda(X)$-modules, while $\Dqcohri(X/G,\Lambda)$ is the $\infty$-category of smooth $G$-representations on such modules (see \cite[Lemma 3.4.26]{mann-p-adic-6-functors}). Moreover, $f_*$ computes smooth $G$-cohomology and $f^*$ associates to every $M \in \Dqcohri(\Lambda(X))$ the trivial $G$-representation on $M$.

We first prove (i). By \cite[Lemma 3.4.19]{mann-p-adic-6-functors} the collection of functors $R\Gamma_\sm(H, -)\colon \Dqcohri(X/G,\Lambda) \to \Dqcohri(X,\Lambda)$ for open compact subgroups $H$ is conservative. Therefore the natural map $f^* \tensor f^! \Lambda^a \to f^!$ is an equivalence if and only if for all compact open subgroups $H \subset G$ the map $R\Gamma_\sm(H, f^* \tensor f^! \Lambda^a) \to R\Gamma_\sm(H, f^!)$ is an equivalence of endofunctors of $\Dqcohri(X,\Lambda)$. Let us denote $f_H\colon */H \to *$ the projection, which factors over the finite étale projection $*/H \to */G$. Then we see that $f^* \tensor f^! \Lambda^a \to f^!$ is an equivalence if and only if for all compact open subgroups $H \subset G$ the map
\begin{align*}
    f_{H*} (f_H^* \tensor f_H^! \Lambda^a) \to f_{H*} f_H^!
\end{align*}
is an equivalence. By \cref{rslt:finite-cd-p-implies-classifying-stack-is-pfine} we have $f_{H*} = f_{H!}$, so in particular $f_{H*}$ satisfies the projection formula. We compute
\begin{align*}
    f_{H*} (f_H^* \tensor f_H^! \Lambda^a) &= - \tensor f_{H*} f_H^! \Lambda^a = - \tensor \IHom(f_{H*}\Lambda^a, \Lambda^a),\\
    f_{H*} f_H^! &= f_{H*} \IHom(\Lambda^a, f_H^!) = \IHom(f_{H!}\Lambda^a, -) = \IHom(f_{H*}\Lambda^a, -).
\end{align*}
It follows that $f^* \tensor f^! \Lambda^a \to f^!$ is an equivalence if and only if $R\Gamma_\sm(H, \Lambda(X)^a) = f_{H*} \Lambda^a$ is dualizable in $\Dqcohri(X,\Lambda)$ for all open compact $H \subset G$. Now assume that $G$ satisfies condition (i) of \cref{def:profinite-Poincare-group}. Then $R\Gamma_\sm(H, \F_p)$ is dualizable in $\D_\solid(\F_p)$ for all $H$. If $p = 0$ in $\Lambda(X)$ then the same follows for $R\Gamma_\sm(H, \Lambda(X)^a) = R\Gamma_\sm(H, \F_p) \tensor_{\F_p} \Lambda(X)^a$. For general $\Lambda$ we can argue as in the proof of \cite[Lemma 3.8.4.(i)]{mann-p-adic-6-functors}. Conversely, assume that $f^* \tensor f^! \Lambda^a \isoto f^!$ is an isomorphism for all $(X,\Lambda)$. Pick $X = \Spa C$ for some algebraically closed field over $\Q_p$ and $\Lambda = \ri_C/p$. Then the Riemann-Hilbert correspondence from \cite[Theorem 3.9.23]{mann-p-adic-6-functors} implies that the base-change functor $RH\colon \D^\sm(G,\F_p) \injto \D^\sm(G,\ri^a_C/p)^\varphi$ is fully faithful and induces an equivalence of dualizable objects. On the other hand, this base-change functor sends $R\Gamma_\sm(H,\F_p)$ to $R\Gamma_\sm(H,\ri^a_C/p)$, hence it follows from the above that $R\Gamma_\sm(H,\F_p)$ is dualizable for all open compact subgroups $H \subset G$. But every $G$-representation $M$ on a finite static $\F_p$-vector space is a finite colimit of the generating representations $\F_p[G/H]$, hence also $R\Gamma_\sm(G,M)$ is dualizable. In particular all $H^i(G,M)$ are finite (here we implicitly use $R\Gamma(G,M) = R\Gamma_\sm(G,M)$ since $M$ is discrete).

We now prove (ii), so from now on we can assume that $G$ satisfies the equivalent conditions of (i). We then need to see that $G$ is $p$-Poincaré if and only if for every strictly totally disconnected $(X,\Lambda) \in \vStacksCoeff$ with $\Lambda = \ri^+_{X^\flat}/\pi$ for some pseudouniformizer $\pi$ on $X^\flat$ the object $f^!\Lambda^a \in \Dqcohri(X/G,\Lambda)$ is invertible. Given $(X,\Lambda)$, let us denote $q\colon X \to X/G$ the projection. Since invertibility satisfies descent, $f^! \Lambda^a$ is invertible if and only if $q^* f^! \Lambda^a$ is invertible. For every compact open subgroup $H \subset G$ we denote $f_H\colon */H \to *$ and $q_H\colon */H \to */G$ the natural maps and we compute (using \cite[Proposition 3.3.9.(iii)]{mann-p-adic-6-functors} and \cref{rslt:finite-cd-p-implies-classifying-stack-is-pfine})
\begin{align*}
    q^* f^! \Lambda^a = f_* q_* q^* f^! \Lambda^a = \varinjlim_H f_* q_{H*} q_H^* f^! \Lambda^a = \varinjlim_H f_{H*} f_H^! \Lambda^a = \varinjlim_H \IHom(f_{H*} \Lambda^a, \Lambda^a).
\end{align*}
By the proof of (i) we know that $f_{H*} \Lambda^a = R\Gamma(H, \Lambda(X)^a)$ is dualizable in $\Dqcohri(X,\Lambda)$ and similarly $R\Gamma(H,\F_p)$ is dualizable in $\D_\solid(\F_p)$. It follows that
\begin{align*}
    \IHom(f_{H*} \Lambda^a, \Lambda^a) = \IHom(\Gamma(H,\F_p), \F_p) \tensor_{\F_p} \Lambda(X)^a.
\end{align*}
Therefore we obtain
\begin{align*}
    q^* f^! \Lambda^a = \varinjlim_H \IHom(f_{H*} \Lambda^a, \Lambda^a) = (\varinjlim_H \IHom(\Gamma(H,\F_p), \F_p)) \tensor_{\F_p} \Lambda(X)^a
\end{align*}
We have $R\Gamma(H,\F_p) = \bigdsum_i H^i(H, \F_p)[-i]$ and since there are no higher $\Ext$-groups between discrete $\F_p$-vector spaces we even have a decomposition of diagrams $(R\Gamma(H,\F_p)_H = \bigdsum_i (H^i(H,\F_p))_H$. On the other hand, we know that if $(M_i)_i$ is any cofiltered system of finite discrete $\F_p$-vector spaces then $\varprojlim_i M_i$ is concentrated in degree $0$ and $\IHom(\varprojlim_i M_i, \F_p) = \varinjlim_i \IHom(M_i, \F_p)$. Altogether this implies
\begin{align*}
    \varinjlim_H \IHom(R\Gamma(H,\F_p), \F_p) &= \IHom(\varprojlim_H R\Gamma(H,\F_p), \F_p) = \IHom(\varprojlim_H \Gamma(G, \F_p[G/H]), \F_p) = \\&= \IHom(R\Gamma(G, \F_{p\solid}[G]), \F_p).
\end{align*}
We deduce:
\begin{align*}
    q^* f^! \Lambda^a = \IHom(R\Gamma(G, \F_{p\solid}[G]), \F_p) \tensor_{\F_p} \Lambda(X)^a.
\end{align*}
In particular, if $G$ is $p$-Poincaré then $q^* f^! \Lambda^a$ is invertible. Conversely, assume that $*/G \to *$ is $p$-cohomologically smooth. Choosing $X = \Spa C$ for some algebraically closed field $C$, the Riemann-Hilbert correspondence (see \cite[Theorem 3.9.23]{mann-p-adic-6-functors}) tells us that the invertibility of $q^* f^! \Lambda^a$ implies that $\IHom(R\Gamma(G, \F_{p\solid}[G]), \F_p)$ is invertible. But then also $R\Gamma(G,\F_{p\solid}[G])$ is invertible and thus $G$ is $p$-Poincaré.
\end{proof}

So far we have only dealt with profinite groups of bounded $p$-cohomological dimension. However, the above results generalize in a straightforward way to \emph{locally} profinite groups satisfying a much weaker cohomological dimension condition. The precise definition is as follows.

\begin{defn} \label{def:locally-profinite-ppoicare}
Let $G$ be a locally profinite group.
\begin{defenum}
    \item We say that $G$ has \emph{locally finite $p$-cohomological dimension} if there is some open subgroup $H \subseteq G$ such that $\cd_p H < \infty$.
    
    \item We say that $G$ is \emph{virtually $p$-Poincaré (of dimension d)} if there is an open compact subgroup $H \subset G$ such that $\cd_p H < \infty$ and $H$ is $p$-Poincaré of dimension $d$ in the sense of Definition \ref{def:profinite-Poincare-group}.
\end{defenum}
\end{defn}

In the following, whenever we talk about a ``virtual $p$-Poincaré group $G$'' we implicitly assume that $G$ has locally finite $p$-cohomological dimension.

\begin{thm} \label{rslt:locally-p-Poincare-equiv-p-cohomo-smoothness}
Let $G$ be a locally profinite group which has locally finite $p$-cohomological dimension. Then the map $*/G \to *$ is $p$-fine, and it is $p$-cohomologically smooth (of pure dimension $d/2$) if and only if $G$ is virtually $p$-Poincaré (of dimension $d$).
\end{thm}
\begin{proof}
This follows immediately from \cref{rslt:p-Poincare-equiv-p-cohom-smoothness} by observing that being $p$-fine and being $p$-cohomologically smooth can both be checked étale locally on the source (note that if $H \subseteq G$ is an open subgroup then the map $*/H \surjto */G$ is an étale cover).
\end{proof}

As a formal consequence of \cref{rslt:locally-p-Poincare-equiv-p-cohomo-smoothness} we obtain the following $p$-adic analog of \cite[Proposition 24.2]{etale-cohomology-of-diamonds}:

\begin{cor} \label{rslt:p-adic-smoothness-criterion-for-quotient}
Let $G$ be a locally profinite group of locally finite $p$-cohomological dimension acting on a small v-stack $Y$ and let $f\colon Y \to X$ be a $G$-equivariant separated bdcs map of small v-stacks, where $G$ acts trivially on $X$.
\begin{corenum}
    \item The map $f'\colon Y/G \to X$ is $p$-fine.
    \item If $f$ is $p$-cohomologically smooth and $G$ is virtually $p$-Poincaré, then $f'$ is $p$-cohomologically smooth.
\end{corenum}
\end{cor}
\begin{proof}
Factor $f'$ as $Y/G \to X/G \to X$. We claim that the first map, i.e. $f/G\colon Y/G \to X/G$, is bdcs: Base-changing this map along the pro-étale cover $X \surjto X/G$ gives $f$, hence the map $Y/G \to X/G$ is separated by \cite[Proposition 10.11.(ii)]{etale-cohomology-of-diamonds}, compactifiable by \cite[Proposition 22.3.(iii)]{etale-cohomology-of-diamonds} (note that $f$ is automatically compactifiable by \cite[Proposition 22.3.(v)]{etale-cohomology-of-diamonds}), $p$-bounded by \cite[Lemma 3.5.10.(ii)]{mann-p-adic-6-functors} and representable in locally spatial diamonds by \cite[Proposition 13.4.(iv)]{etale-cohomology-of-diamonds}. Since $X/G \to X$ is $p$-fine by \cref{rslt:finite-cd-p-implies-classifying-stack-is-pfine} we deduce that $f'$ is $p$-fine. This proves (i).

To prove (ii), assume in addition that $f$ is $p$-cohomologically smooth and $G$ is virtually $p$-Poincaré. Then it follows from \cite[Lemma 3.8.5.(ii)]{mann-p-adic-6-functors} that $f/G$ is $p$-cohomologically smooth, and it follows from \cref{rslt:locally-p-Poincare-equiv-p-cohomo-smoothness} that the map $X/G \to X$ is $p$-cohomologically smooth. Together this implies that $f'$ is $p$-cohomologically smooth, as desired.
\end{proof}

 We end this subsection by providing examples of virtual $p$-Poincaré groups. The key result for applications is the following theorem. 
 \begin{thm}\label{padicLiePoincare}Any $p$-adic Lie group is virtually $p$-Poincaré.
 \end{thm}
 
 By the local nature of Definition \ref{def:locally-profinite-ppoicare} and standard structural results on $p$-adic Lie groups, this is an immediate consequence of the following theorem, which is what we shall actually prove.
 
 \begin{thm}\label{thm:uniformly-powerful-ppoincare}Any torsion-free compact $p$-adic analytic group is $p$-Poincaré.
 \end{thm}
 
The proof relies on deep results of Lazard \cite{lazard-monster}, which are in some sense the only ``non-formal" input we need. To apply Lazard's results, we need some general comparison between condensed group cohomology and classical continuous group cohomology, which takes into account a natural topological structure on the latter. To state the relevant result, fix a locally profinite group $G$. Let $V$ be a continuous representation of $G$ on a Hausdorff locally compact  $\F_p$-module. Then $\underline{V}$ is a solid $\F_p$-vector space with continuous $G$-action, and in particular is an $\F_{p\solid}[G]$-module. Thus we may form the condensed group cohomology \[R\Gamma(G,\underline{V})=\IHom_{\F_{p\solid}[G]}(\F_{p},\underline{V}) \in \mathcal{D}_{\solid}(\F_{p}) \]
as considered above.

On the other hand, the classical continuous group cohomology $R\Gamma_{\mathrm{cts}}(G,V)$ is by definition computed by the usual complex \[0 \to V \to \mathcal{C}(G,V) \to \mathcal{C}(G^2,V) \to \cdots  \]
of inhomogeneous cochains. Giving each $\mathcal{C}(G^n,V)$ the compact-open topology, this is naturally a complex of Hausdorff locally compact $\F_p$-modules. In particular, passing termwise to the condensation, we get a complex \[0 \to \underline{V} \to \underline{\mathcal{C}(G,V)} \to \underline{\mathcal{C}(G^2,V)} \to \cdots  \]
of solid $\F_p$-vector spaces. Let $\underline{R\Gamma_{\mathrm{cts}}(G,V)}$ denote the image of this complex in $\mathcal{D}_{\solid}(\F_{p})$.

\begin{lem}\label{lem:internal-group-coh-comparison}There is a natural isomorphism
\[ R\Gamma(G,\underline{V}) \simeq \underline{R\Gamma_{\mathrm{cts}}(G,V)} \]
in $\mathcal{D}_{\solid}(\F_{p})$.
\end{lem}

\begin{proof}[Proof of Theorem \ref{thm:uniformly-powerful-ppoincare}] Let $G$ be as in the statement of the theorem. Then condition (i) in Definition \ref{def:profinite-Poincare-group} is classical, see e.g. \cite[\S 5]{symonds-weigel-cohomology}. For condition (ii), recall that $R\Gamma_{\mathrm{cts}}(G,\F_p[[G]]) \simeq \F_p[-\dim G]$ by \cite[Ch. V]{lazard-monster} and \cite[\S 4-5]{symonds-weigel-cohomology}. Since $\underline{\F_p[[G]]} =\F_{p \solid}[G]$, the conclusion now follows from Lemma \ref{lem:internal-group-coh-comparison}.
\end{proof}

\begin{proof}[Proof of Lemma \ref{lem:internal-group-coh-comparison}] Arguing as in the proof of \cite[Proposition 3.4.6]{mann-p-adic-6-functors}, one sees that $R\Gamma(G,\underline{V})$ is computed as the cohomology of the complex
\[0 \to \underline{V} \to \underline{\Hom}(\underline{G},\underline{V}) \to \underline{\Hom}(\underline{G}^2,\underline{V}) \to \cdots\]
where the internal homs are taken in condensed sets. By the definition of $\underline{R\Gamma_{\mathrm{cts}}(G,V)}$, it now suffices to show that $\underline{\Hom}(\underline{S},\underline{V})$ identifies with $\underline{\cts(S,V)}$ functorially in locally profinite $S$. This is a special case of the next lemma.
\end{proof}

\begin{lem}Let $X$ be any locally compact Hausdorff space, and let $Y$ be any compactly generated T1 topological space. Equip the space of continuous maps $\cts(X,Y)$ with the compact-open topology. Then $\cts(X,Y)$ is compactly generated T1, and $\underline{\cts(X,Y)} = \underline{\Hom}(\underline{X},\underline{Y})$, where $\underline{\Hom}$ denotes the internal hom in condensed sets. 
\end{lem}
\begin{proof}The claim regarding the topology on $\cts(X,Y)$ is easy. For the remaining claim, let $T$ be any profinite space. We then compute that
\begin{align*}
    \Hom_{\mathrm{Cond}}(\underline{T}, \underline{\Hom}(\underline{X},\underline{Y})) &= \Hom_{\mathrm{Cond}}(\underline{T} \times \underline{X},\underline{Y}) \\
      &= \Hom_{\mathrm{Top}}(T \times X, Y) \\
      &= \Hom_{\mathrm{Top}}(T, \cts(X,Y)) \\
      &= \Hom_{\mathrm{Cond}}(\underline{T}, \underline{\cts(X,Y)})
\end{align*}
naturally in $T,X,Y$. The first equality is immediate from the universal property of $\underline{\Hom}$. The second equality follows from the fact that on compactly generated T1 spaces, $S \mapsto \underline{S}$ is fully faithful and commutes with finite products. The fourth equality follows similarly, using the first claim in the lemma. The third equality follows from the assumption that $X$ is locally compact Hausdorff, which is exactly the condition needed to ensure that $\cts(X,-)$ with its compact-open topology is right-adjoint to $- \times X$. Since $T$ is arbitrary, the result now follows from Yoneda. \end{proof}

\subsection{Admissible representations}

Given a locally profinite group $G$, an important subclass of all smooth $G$-representations are the admissible ones. Classically, a smooth $G$-representation on an $\F_p$-vector space $V$ is called admissible if for a basis of compact open subgroups $H \subseteq G$ the space $V^H$ of $H$-invariants is finite. In the condensed derived world the correct analogue of this definition seems to be the following:

\begin{defn}
Let $G$ be a locally profinite group which has locally finite $p$-cohomological dimension. A smooth $G$-representation $V \in \D^\sm_\solid(G,\F_p)$ is called \emph{admissible} if for a basis of compact open subgroups $H \subseteq G$ the smooth group cohomology $R\Gamma_\sm(H,V) \in \D_\solid(\F_p)$ is perfect.
\end{defn}

\begin{rmks} \label{admissibleremarks}
\begin{rmksenum}
    \item An admissible $G$-representation $V \in \D^\sm_\solid(G,\F_p)$ is automatically discrete because all $R\Gamma_\sm(H,V)$ are discrete by assumption and hence so is $V = \varinjlim_H R\Gamma_\sm(H,V)$ (cf. \cite[Lemma 3.4.19]{mann-p-adic-6-functors}). Thus $V$ admits a representation as a complex of classical smooth $G$-representations.
    
    \item Suppose that $G$ is a $p$-adic Lie group. Then by \cite[Corollary 4.12]{schneider-sorensen-dual-reps} a smooth representation $V \in \D^\sm_\solid(G,\F_p)$ is admissible if and only if it is bounded and discrete and all $\pi_n V$ are admissible smooth representations in the classical sense.
\end{rmksenum}
\end{rmks}

Now suppose we are additionally given an algebraically closed perfectoid field $C$ with pseudouniformizer $\pi | p$ and we consider the classifying stack $B\underline{G}=\Spd(C)/\underline{G}$. Then $\DqcohriX {B\underline{G}}^\varphi$ is the $\infty$-category of smooth $G$-representations on solid almost $\varphi$-modules over $\ri^a_C/\pi$ and by the $p$-torsion Riemann-Hilbert correspondence this $\infty$-category contains the $\infty$-category $\D^\sm(G,\F_p)$ of smooth $G$-representations as a full subcategory. We make the following definition:

\begin{defn} \label{def:admissible-rep-on-classifying-stack}
Let $G$ be a locally profinite group which has locally finite $p$-cohomological dimension. Let $C$ be an algebraically closed perfectoid field with pseudouniformizer $\pi \in \ri_C$ such that $\pi | p$. A smooth $G$-representation $M \in \D^{\sm,a}_\solid(G, \ri_C/\pi)^\varphi$ is called \emph{admissible} if it comes from an admissible $\F_p$-representation via the embedding
\begin{align*}
    - \tensor_{\F_p} \ri^a_C/\pi\colon \D^\sm(G,\F_p) \injto \D^{\sm,a}_\solid(G,\ri_C/\pi)^\varphi    
\end{align*}
obtained from the Riemann-Hilbert correspondence on the classifying stack $B\underline{G}$.
\end{defn}

One might also consider defining $M$ to be admissible if the smooth group cohomology groups $R\Gamma_\sm(H, M)$ are perfect for a basis of open subgroups $H \subseteq G$. In fact, this easily turns out to be equivalent to \cref{def:admissible-rep-on-classifying-stack}, as we will see below. If $G$ is $p$-Poincaré then there is also a characterization of admissible representations in terms of their behavior under duality. To show this, we need the following characterization of perfect (equivalently, dualizable) objects in terms of their duals:

\begin{lem} \label{rslt:discrete-dual-plus-reflexive-implies-dualizable-over-perfd-field}
Let $K$ be a perfectoid field with pseudouniformizer $\pi \in \ri_K$. Then a $\varphi$-module $M \in \Dqcohri(\ri_K/\pi)^\varphi$ is dualizable if and only if $M$ is bounded, both $M$ and $M^\vee$ are discrete and the natural map $M \to M^{\vee\vee}$ is an isomorphism.
\end{lem}
\begin{proof}
First assume that $M$ is dualizable. Then by \cite[Proposition 3.7.5]{mann-p-adic-6-functors} both $M$ and $M^\vee$ are discrete. It follows from \cite[Lemmas 3.9.2, 3.9.3]{mann-p-adic-6-functors} that $M$ is bounded. The isomorphism $M \isoto M^{\vee\vee}$ is clear. This proves the ``only if'' part of the claim.

We now prove the ``if'' part, so assume that $M$ is bounded, $M$ and $M^\vee$ are discrete and $M \isoto M^{\vee\vee}$ is an isomorphism. We want to show that $M$ is dualizable, for which we can assume that $M$ is concentrated in homological degree $\ge 0$. Now consider $M_0 := \pi_0 M$, which is a static discrete $\ri^a_K$-module. From the fact that $M^\vee$ is discrete it follows that $\pi_0 M_0^\vee = \pi_0 M^\vee$ is discrete. After choosing a surjection $\bigdsum_I \ri^a_K/\pi \surjto M_0$ we get an injection $\pi_0 M_0^\vee \injto \prod_I \ri^a_K/\pi$. Now apply $\pi_0 (-)_*$ to get a submodule
\begin{align*}
    M_0' := \pi_0 (\pi_0 M_0^\vee)_* \subseteq \prod_I A,
\end{align*}
where we abbreviate $A := \pi_0 (\ri^a_K/\pi)_*$. Note that $M_0'$ is equipped with the subspace topology since it is a kernel of a map $\prod_I A \to \prod_J A$. Moreover, the fact that $\pi_0 M_0^\vee$ is discrete implies that for every $\varepsilon \in \mm_K$ the submodule $T_\varepsilon \subset M_0'$ of $\varepsilon$-torsion in $M_0'$ is open. In particular there is an open subset $U_\varepsilon \subset \prod_I A$ containing $0$ such that $U_\varepsilon \isect M_0' \subseteq T_\varepsilon$. In fact we can choose $U_\varepsilon$ to be an element of the standard basis for the product topology, i.e. $U_\varepsilon = U_\varepsilon' \times \prod_{I \setminus I_\varepsilon} A$ for some finite subset $I_\varepsilon \subset I$ and some open subset $U_\varepsilon' \subset \prod_{I_\varepsilon} A$. This means that if $(a_i)_i, (b_i)_i \in \prod_I A$ lie in $M_0'$ and $a_i = b_i$ for $i \in I_\varepsilon$ then $\varepsilon (a_i)_i = \varepsilon (b_i)_i$. Now let $M_{0,\varepsilon}' \subset \prod_{I_\varepsilon} A$ be the image of $M_0'$ under the projection $\prod_I A \surjto \prod_{I_\varepsilon} A$ (equipped with the subspace topology). Then there is a natural projection map $f'_\varepsilon\colon M_0' \to M_{0,\varepsilon}'$. We can also construct a map $g'_\varepsilon\colon M_{0,\varepsilon}' \to M_0'$ as follows: Given $a \in M_{0,\varepsilon}'$ we pick any $a' \in M_0'$ such that $f'_\varepsilon(a') = a$. Then we set $g'_\varepsilon(a) := \varepsilon a'$, which by the above observation is independent of the choice of $a'$. Note that $f'_\varepsilon g'_\varepsilon = \varepsilon$ and $g'_\varepsilon f'_\varepsilon = \varepsilon$, hence we have $M'_0 \approx_\varepsilon M'_{0,\varepsilon}$. Going back to the almost world we now see that
\begin{align*}
    \pi_0 M_0^\vee \approx_\varepsilon M'^a_{0,\varepsilon}, \qquad \text{with $M'^a_{0,\varepsilon} \subseteq \prod_{i=1}^{n_\varepsilon} \ri^a_K/\pi$ a discrete submodule},
\end{align*}
where we took the liberty of naming $I_\varepsilon = \{ 1, \dots, n_\varepsilon \}$. By \cite[Lemma 3.7.15]{mann-p-adic-6-functors} the discrete dualization functor $N \mapsto (N^\vee)_\omega$ is $t$-exact on discrete $N$. By applying this functor to $M^\vee$ it follows from the isomorphism $M \isoto M^{\vee\vee}$ that also $M_0 \isoto ((\pi_0 M_0^\vee)^\vee)_\omega$ is an isomorphism. Letting $M_{0,\varepsilon} := ((M'^a_{0,\varepsilon})^\vee)_\omega$ we deduce that $M_0 \approx_{\varepsilon} M_{0,\varepsilon}$ and there is a surjection $\bigdsum_{i=1}^{n_\varepsilon} \ri^a_K/\pi \surjto M_{0,\varepsilon}$. As this works for all $\varepsilon$, we deduce that $M_0$ is almost finitely generated as an $\ri^a_K/\pi$-module. Then it is also almost finitely generated as an $\ri^a_K$-module (where we equip $\ri^a_K$ with the discrete topology in this proof), hence by \cite[Theorem 2.5]{rigid-p-adic-hodge} we have
\begin{align*}
    M_0 \approx (\ri^a_K/\pi)^r \dsum \ri^a_K/I_{\gamma_1} \dsum \ri^a_K/I_{\gamma_2} \dsum \dots
\end{align*}
for a series of real numbers $1 > \gamma_1 \ge \gamma_2 \ge \dots \ge 0$ converging to $0$. Here for a real number $\gamma$ we denote by $I_\gamma \subset \ri_K$ the ideal generated by $\pi^{\gamma'}$ for all $\gamma' > \gamma$ (see \cite[Example 2.4.(i)]{rigid-p-adic-hodge}). The $\varphi$-module structure $\varphi^* M \isoto M$ induces an isomorphism $\pi_0 \varphi^* M_0 \isoto M_0$. On the other hand for $\gamma < 1$ we have
\begin{align*}
    \pi_0 \varphi^*(\ri^a_K/I_\gamma) = \begin{cases} \ri^a_K/\pi, \qquad &\text{if $\gamma \ge 1/p$},\\ \ri^a_K/I_{p\gamma}, \qquad &\text{else}, \end{cases}
\end{align*}
as $\ri^a_K/\pi$-modules. This forces all $\gamma_i = 0$, i.e. we have $M_0 \approx (\ri^a_K/\pi)^r$. It follows that $M_0$ is dualizable: We need to check that the natural map $\alpha\colon M_0 \tensor M_0^\vee \to \IHom(M_0, M_0)$ is an isomorphism but since the same statement holds for $(\ri^a_K/\pi)^r$ in place of $M_0$ one deduces easily that both kernel and cokernel of $\pi_i \alpha$ are killed by $\varepsilon$ for all $i \in \Z$ and $\varepsilon \in \mm_K$ (cf. the proof of \cite[Lemma 3.7.17]{mann-p-adic-6-functors}). In particular $M_0$ is flat (which can also be seen directly), so that $\pi_0 \varphi^* M_0 = \varphi^* M$ and therefore the $\varphi$-module structure on $M_0$ restricts to a $\varphi$-module structure on $M_0$. Since the conditions on $M$ are preserved under finite limits, we reduce the claimed dualizability of $M$ to that of $\tau_{\ge1} M$, which satisfies the same assumptions as $M$. Since $M$ is bounded, we can conclude by induction.
\end{proof}

As promised, here come the three equivalent ways of characterizing admissible representations on a classifying stack. The third characterization is very useful for showing that admissible representations are stable under various of the 6 operations, and plays a key role in our analysis of the Jacquet-Langlands functor.

\begin{prop}\label{prop:admissiblecriterion}
Let $C$ be an algebraically closed perfectoid field with pseudouniformizer $\pi \in \ri_C$ such that $\pi | p$ and let $G$ be a locally profinite $p$-Poincaré group. Then for $M \in \D^{\sm,a}_\solid(G, \ri_C/ \pi)^\varphi$ the following are equivalent:
\begin{propenum}
    \item $M$ is bounded and admissible.
    \item $M$ is bounded and for a basis of compact open subgroups $H \subseteq G$, $\Gamma_\sm(H,M)$ is dualizable in $\Dqcohri(\ri_C/\pi)$.
    \item Both $M$ and $M^\vee$ are bounded and discrete and the natural morphism $M \isoto M^{\vee\vee}$ is an isomorphism.
\end{propenum}
If this is the case then $M^\vee$ is bounded and admissible and this dual can unambiguously be computed in any of the following $\infty$-categories: $\D^\sm(G,\F_p)$, $\D^\sm_\solid(G,\F_p)$, $\D^{\sm,a}_\solid(G,\ri_C/\pi)^\varphi_\omega$, $\D^{\sm,a}_\solid(G,\ri_C/\pi)^\varphi$.
\end{prop}
\begin{proof}
Note first that (i) implies (ii) because $R\Gamma_\sm(H,V) \tensor_{\F_p} \ri^a_C/\pi = R\Gamma_\sm(H, V \tensor_{\F_p} \ri^a_C/\pi)$ for all $V \in \D^\sm(G,\F_p)$ (e.g. by the explicit computation of smooth group cohomology in \cite[Corollary 3.4.17]{mann-p-adic-6-functors}). Conversely, assume that (ii) is satisfied. Then $M$ is discrete because $M = \varinjlim_H R\Gamma_\sm(H,M)$, where on the right-hand side $H$ ranges through all compact open subgroups of $G$ (see \cite[Lemma 3.4.19]{mann-p-adic-6-functors}). We claim that $M$ lies in $\D^\sm(G,\F_p)$, i.e. that the natural morphism $M^\varphi \tensor_{\F_p} \ri^a_C \isoto M$ is an isomorphism. By the colimit property we just used the collection of functors $R\Gamma_\sm(H,-)$, where $H$ ranges through a basis of compact open subgroups of $G$, is conservative, hence it is enough to show that $R\Gamma_\sm(H,M^\varphi \tensor_{\F_p} \ri^a_C) = R\Gamma_\sm(H,M)$ for all $H$ in the basis. Noting that that $R\Gamma_\sm(H,-)$ commutes with $(-)^\varphi$ (which is formal) and with $- \tensor_{\F_p} \ri^a_C/\pi$ (as used above),  the claim boils down to showing that for $M_H := R\Gamma_\sm(H,M)$ we have $M_H^\varphi \tensor_{\F_p} \ri^a_C/\pi \isoto M_H$. But by assumption $M_H$ is dualizable, hence the desired equivalence follows from the Riemann-Hilbert correspondence (see \cite[Theorem 3.9.23]{mann-p-adic-6-functors}). From the Riemann-Hilbert correspondence we also deduce that $R\Gamma_\sm(H, M^\varphi)$ is dualizable, hence $M^\varphi$ is admissible, as desired.

We have now finished the proof that (i) and (ii) are equivalent. To prove that they are also equivalent to (iii) we make the following computation. Let $X = \Spa C$, $\Lambda = \ri_C/\pi$, $\omega_G := f^! \Lambda^a$ and for every compact open subgroup $H \subset G$ let $f_H\colon X/H \to X$ be the natural map. Then from \cref{rslt:finite-cd-p-implies-classifying-stack-is-pfine,rslt:p-Poincare-equiv-p-cohom-smoothness} we get:
\begin{align*}
    f_{H*} M^\vee &= f_{H*} \IHom(M, \Lambda^a) = f_{H*} \IHom(M \tensor \omega_G, \omega_G) = \IHom(f_{H*}(M \tensor \omega_G), \Lambda^a)\\
    &= (f_{H*}(M \tensor \omega_G))^\vee.
\end{align*}
Note also that for small enough $H$ the action of $H$ on $\omega_G$ is trivial, so that $- \tensor \omega_G$ is just a shift in the $\infty$-category of $H$-representations. Consequently for small enough $H$ the functor $R\Gamma_\sm(H,-)$ preserves duals up to shift. Since the functors $R\Gamma_\sm(H,-)$ are also conservative (for $H$ varying in a basis of compact open subgroups of $G$) it follows that (iii) is equivalent to the claim that $M$ and $M^\vee$ are bounded and for a basis $H$ of compact open subgroups of $G$ we have that both $M_H := R\Gamma_\sm(H,M)$ and $M_H^\vee$ are discrete and the natural morphism $M_H \isoto M_H^{\vee\vee}$ is an isomorphism. By \cref{rslt:discrete-dual-plus-reflexive-implies-dualizable-over-perfd-field} this implies (ii). Moreover, if (ii) is satisfied then all the conditions in (iii) follow immediately except possibly that $M^\vee$ is also bounded. To see that $M^\vee$ is indeed bounded, note first that $M^\vee = \varinjlim_H R\Gamma_\sm(H, M^\vee)$ (see \cite[Lemma 3.4.19]{mann-p-adic-6-functors}) and by the above computation the right-hand side is (up to a constant shift) isomorphic to $\varinjlim_H M_H^\vee$. On the other hand, the discrete dualization functor $N \mapsto (N^\vee)_\omega$ is $t$-exact on discrete objects in $\Dqcohri(\ri_C/\pi)$ (see \cite[Lemma 3.7.15]{mann-p-adic-6-functors}). In particular, if $M_H$ is concentrated in homological degrees $[a,b]$ then $M_H^\vee$ is concentrated in homological degrees $[-b,-a]$. The fact that $M$ is bounded implies that we can choose $a$ and $b$ independently of $H$, which then implies that $M^\vee$ is also bounded.

It remains to prove the last part of the claim, so assume that $M$ is admissible and bounded. It follows immediately from (iii) that also $M^\vee$ (computed in $\D^{\sm,a}_\solid(G,\ri_C/\pi)$) is admissible (and bounded), so in particular it is discrete and comes from an $\F_p$-representation via the Riemann-Hilbert correspondence. It is thus formal that this dual can be computed in any of the claimed $\infty$-categories.
\end{proof}

\begin{rmk} \label{rmk:do-not-need-biduality-map-for-admissibility}
It follows from the proof that in \cref{rslt:discrete-dual-plus-reflexive-implies-dualizable-over-perfd-field} it is enough to require that there is \emph{some} isomorphism $M \isom M^{\vee\vee}$ instead of insisting that the natural map $M \to M^{\vee\vee}$ is an isomorphism. It follows that in condition (iii) of \cref{prop:admissiblecriterion} it is enough to require the existence of \emph{some} isomorphism $M \isom M^{\vee\vee}$.
\end{rmk}

\section{Duality and applications}

\subsection{The main result}

Fix a finite extension $F/\Q_p$, let $G := \GL_n(F)$, let $D$ be the central division algebra over $F$ of invariant $1/n$ and let $C$ be the completed algebraic closure of $F$. In the following we will work with small v-stacks over $\Spa C$, so in particular the final object $*$ will denote $\Spa C$ and $B\underline H = [\Spa C/\underline H]$ for any locally profinite group $H$. Recall the fundamental diagram from the introduction:
\begin{center}\begin{tikzcd}
    & X \arrow[dl,"f",swap] \arrow[dr,"g"]\\
    B\underline G && B\underline{D^\times}
\end{tikzcd}\end{center}
Here $X$ is obtained from the infinite level Lubin-Tate tower by modding out the commuting actions of $G$ and $D^\times$, so that in particular we have $X \isom [\mathbf P^{n-1}_C/\underline{D^\times}] \isom [\Omega_C^{n-1}/\underline{G}]$, where $\Omega_C^{n-1} = \mathbf P^{n-1}_C \setminus \mathbf P^{n-1}_C(\Q_p)$ denotes Drinfeld space. It follows that $f$ and $g$ are smooth and that $g$ is proper. Recall also that for every locally profinite group $H$ which has locally finite $p$-cohomological dimension (e.g. $H = G$ or $H = D^\times$) we have
\begin{align*}
    \D_\et(B\underline{H}, \F_p) &= \D^\sm(H, \F_p),\\
    \Dqcohri(\ri^+_{B\underline H}/p)^\varphi &= \Dcatrepsldasm H{\ri_C/p}^\varphi,
\end{align*}
and the upper $\infty$-category embedds fully faithfully into the lower one via the Riemann-Hilbert functor. As explained in the introduction, we can now define the Jacquet-Langlands functor as follows.

\begin{defn}
In the above setup, we define the \emph{$p$-adic Jacquet-Langlands functors}
\begin{align*}
    \mathcal J &\colon \D^\sm(G, \F_p) \to \D^\sm(D^\times, \F_p), && A \mapsto g_{\et,*} f_\et^* A,\\
    \mathcal J_\solid &\colon \Dcatrepsldasm G{\ri_C/p}^\varphi \to \Dcatrepsldasm{D^\times}{\ri_C/p}^\varphi, && A \mapsto g_* f^* A,
\end{align*}
where we implicitly use the identifications of sheaves on the classifying stack with smooth representations discussed above.
\end{defn}

The functor $\mathcal J_\solid$ is defined in the setting of a full 6-functor formalism and should thus be seen as the ``actual'' Jacquet-Langlands functor, whereas $\mathcal J$ is a more naively defined shadow of $\mathcal J_\solid$. In fact, it is completely formal that for all $A \in \D^\sm(G, \F_p)$ we have
\begin{align}
    \mathcal J(A) = (\mathcal J_\solid(A \tensor \ri^{+a}_{B\underline G}/p))^\varphi. \label{eq:naive-J-is-phi-inv-of-solid-J}
\end{align}
In other words, $\mathcal J(A)$ remembers the $\varphi$-invariants of the Jacquet-Langlands correspondent $\mathcal J_\solid(A \tensor \ri^+_{B\underline G}/p)$ of $A$. It turns out that if $A$ is admissible then in fact these $\varphi$-invariants capture the whole Jacquet-Langlands correspondent. In fact, the Jacquet-Langlands correspondent of an admissible representation is itself admissible and hence lies in the image of the Riemann-Hilbert functor, as is shown in the following result.

\begin{thm}\label{mainthm}
In the above setup, let $A \in \D^\sm(G,\F_p)$ be admissible. Then:
\begin{thmenum}
    \item $\mathcal J(A) \in \D^\sm(D^\times, \F_p)$ is admissible.
    
    \item The natural primitive comparison map
    \begin{align*}
        \mathcal J(A) \tensor \ri^{+a}_{B\underline{D^\times}}/p \isoto \mathcal J_\solid(A \tensor \ri^{+a}_{B\underline G}/p)
    \end{align*}
    induced by \cref{eq:naive-J-is-phi-inv-of-solid-J} is an isomorphism.
    
    \item There is a natural isomorphism
    \begin{align*}
        (\mathcal{S}_{D^\times} \circ \mathcal{J})(A) = (\mathcal{J} \circ \mathcal{S}_G)(A)[2n-2](n-1)
    \end{align*}
    in $\D^\sm(D^\times,\F_p)$.
\end{thmenum}
\end{thm}

Recall that Scholze already proved (i) and (ii). We will prove (i)--(iii) simultaneously, independently of Scholze's results. In fact, we will see that all (i)--(iii) follow formally from the following unconditional duality for $\mathcal J_\solid$:

\begin{lem}\label{solidduality}
In the above setup, there is a natural isomorphism
\begin{align*}
    (-)^\vee \circ \mathcal{J}_\solid = (\mathcal{J}_\solid \circ (-)^\vee)[2n-2](n-1)
\end{align*}
of functors $\Dcatrepsldasm G{\ri_C/p}^\varphi \to \Dcatrepsldasm{D^\times}{\ri_C/p}^\varphi$.
\end{lem}
\begin{proof}
For every $A \in \Dcatrepsldasm G{\ri_C/p}^\varphi$ we get the following chain of natural equivalences (using that $g$ is proper and smooth and that $f$ is smooth):
\begin{align*}
    \IHom(g_* f^* A, \ri^{+a}_{B\underline{D^\times}}/p) &= \IHom(g_! f^* A, \ri^{+a}_{B\underline{D^\times}}/p)\\
    &= g_* \IHom(f^* A, g^!(\ri^{+a}_{B\underline{D^\times}}/p))\\
    &= g_* \IHom(f^* A, \ri^{+a}_X/p)[2n-2](n-1)\\
    &= g_* f^* \IHom(A, \ri^{+a}_{B\underline G}/p)[2n-2](n-1).
\end{align*}
In the last step we used that by the smoothness of $f$, $f^*$ commutes with internal hom (see \cite[Proposition 3.8.7]{mann-p-adic-6-functors}.
\end{proof}

\begin{proof}[Proof of Theorem \ref{mainthm}]
Fix an admissible $A \in \D^\sm(G,\F_p)$ as in the claim and let us denote $A' := A \tensor \mathcal{O}_{B\underline{G}}^{+a}/p$ the corresponding object in $\Dcatrepsldasm G{\ri_C/p}^\varphi$ given by the Riemann-Hilbert embedding. Then $A'$ is admissible (by definition) and hence by \cref{prop:admissiblecriterion} both $A'$ and $A'^\vee$ are bounded and discrete and the natural map $A' \isoto A'^{\vee\vee}$ is an isomorphism. Clearly $\mathcal J_\solid$ preserved bounded and discrete objects (this is true for $f^*$ like for any pullback and it follows for $g_*$ from the fact that $g$ is proper and $p$-bounded). Thus it follows from \cref{solidduality} that both $\mathcal J_\solid(A')$ and $\mathcal J_\solid(A')^\vee$ are bounded and discrete. Moreover, \cref{solidduality} also implies that there is an isomorphism $\mathcal J_\solid(A') \isom \mathcal J_\solid(A')^{\vee\vee}$, so altogether it follows from \cref{prop:admissiblecriterion} (and \cref{rmk:do-not-need-biduality-map-for-admissibility}) that $\mathcal J_\solid(A')$ is admissible. This immediately implies (i). It also implies (ii) because we now know that $\mathcal J_\solid(A')$ lies in the image of the Riemann-Hilbert functor. Now (iii) follows from \cref{solidduality} and the last part of \cref{prop:admissiblecriterion}.
\end{proof}

\begin{rmk}In general, both sides of the duality isomorphism are somewhat intractible. However, in certain favorable situations it reduces to a more concrete statement. We record one useful instance of this in the following proposition. Recall from Proposition \ref{dimbasics} that for any admissible smooth $\pi \in \mathrm{Rep}_{\mathbf{F}_{p}}^{\mathrm{sm}}(G)$, the derived dual $\mathcal{S}_{G}^{\dim_G \pi}(\pi)$ is always nonzero. Following \cite{kohlhaase}, we say $\pi$ is Cohen-Macaulay if $\mathcal{S}_{G}^{i}(\pi)=0$ for all $i\neq \dim_G \pi$.
\end{rmk}
\begin{prop}\label{dualitysimple}
Let $\pi \in \mathrm{Rep}_{\mathbf{F}_{p}}^{\mathrm{sm}}(G)$ be an admissible representation which is Cohen-Macaulay, and such that there is an integer $c$ with $\mathcal{J}^i(\pi)=0$ for all $i\neq c$. Then there are isomorphisms
\[\mathcal{S}_{D^\times}^i(\mathcal{J}^c(\pi)) \simeq \mathcal{J}^{i+2n-2-c-\dim_G \pi}(\check{\pi})
\]
for all $i$, where $\check{\pi}=\mathcal{S}_{G}^{\dim_G \pi}(\pi)$ is the unique nonzero derived dual of $\pi$.
\end{prop}

\begin{proof}Evaluate both sides of the duality isomorphism $\mathcal{S}_{D^\times} \circ \mathcal{J}) \cong (\mathcal{J} \circ \mathcal{S}_G)[2n-2](n-1)$ on $\pi$. Since $\mathcal{J}(\pi)=\mathcal{J}^{c}(\pi)[-c]$ by assumption, the left-hand side evaluates to $\mathcal{S}_{D^\times}(\mathcal{J}^{c}(\pi)[-c]) \cong \mathcal{S}_{D^\times}(\mathcal{J}^{c}(\pi))[c]$. Next, since $\mathcal{S}_G(\pi)=\check{\pi}[-\dim_G \pi]$ by assumption, the right-hand side simplifies to $\mathcal{J}(\check{\pi}[-\dim_G \pi])[2n-2](n-1) \cong \mathcal{J}(\check{\pi})[2n-2-\dim_G \pi](n-1)$. Equating both sides and rearranging the shifts, we get $\mathcal{S}_{D^\times}(\mathcal{J}^{c}(\pi)) = \mathcal{J}(\check{\pi})[2n-2-\dim_G \pi-c](n-1)$. Passing to cohomology on both sides (and ignoring the Tate twists), we get the result.
\end{proof}

\begin{rmk}\label{dualitygeneralcoefficients}We explain how to show a variant of our duality isomorphism in the setting of $\mathbf{Z}/p^n$ or $\mathbf{Z}_p$-coefficients. We explain the first case; the second follows from a limit argument. First some preliminary observations. Let $H$ be any $p$-adic Lie group. One easily checks that Proposition \ref{Detclassifying} extends to the setting of $\mathbf{Z}/p^n$-coefficients, yielding a symmetric monoidal equivalence $\D^{\mathrm{sm}}(H,\mathbf{Z}/p^n)\cong \D_{\et}(B\underline{H},\mathbf{Z}/p^n)$. On these categories we again have a duality functor $\mathcal{S}_H(-)$, given by internal hom towards the monoidal unit. We say $A \in \D^{\mathrm{sm}}(H,\mathbf{Z}/p^n)$ is admissible if $R\Gamma_{\sm}(H,A)$ is a perfect complex of $\mathbf{Z}/p^n$-modules for a basis of open subgroups $H$. Admissible complexes are automatically bounded, and are preserved by duality and $-\otimes_{\mathbf{Z}/p^n}^{\mathbf{L}} \F_p$, and the latter functor is conservative on admissible complexes. Conversely, if $A$ is given and $A\otimes_{\mathbf{Z}/p^n}^{\mathbf{L}} \F_p \in \D^{\mathrm{sm}}(H,\F_p)$ is admissible, then $A$ is necessarily admissible.

With this in mind, we \emph{define} $\mathcal{J}: \D^{\mathrm{sm}}(G,\mathbf{Z}/p^n) \to \D^{\mathrm{sm}}(D^\times,\mathbf{Z}/p^n)$ as $\mathcal{J} = g_{\et \ast}f^{\ast}_{\et}$. This is compatible with the $\F_p$ version of $\mathcal{J}$ via $-\otimes_{\mathbf{Z}/p^n}^{\mathbf{L}} \F_p$. Using this observation and the final statement in the previous paragraph together with the finiteness properties of $\mathcal{J}$ with $\F_p$-coefficients, it is easy to see that $\mathcal{J}$ preserves admissibility in this setting. Now, the key observation is that there is a natural map $\mathcal{J}(A) \otimes \mathcal{J}(\mathcal{S}_G(A)) \to \mathbf{Z}/p^n[2-2n](1-n)$, defined as the composition
\begin{align*}
    \mathcal{J}(A) \otimes \mathcal{J}(\mathcal{S}_G(A)) & \to  \mathcal{J}(A \otimes \mathcal{S}_G(A)) \\
      &\to \mathcal{J}(\mathbf{Z}/p^n) \\
      &\to \tau^{\geq 2n-2} \mathcal{J}(\mathbf{Z}/p^n) \\
      &\cong \mathbf{Z}/p^n[2-2n](1-n).
\end{align*}
Here the first arrow is obtained from the lax monoidal structure of $\mathcal{J}$, the second arrow is induced by the evaluation map $A \otimes \mathcal{S}_G(A) \to \mathbf{Z}/p^n$, and the final isomorphism is induced by the usual computation of the \'etale cohomology of $\mathbf{P}^{n-1}$. By adjunction (and rearranging the shifts and twists), this map induces a map $\mathcal{J}(\mathcal{S}_G(A))[2n-2](n-1) \to \mathcal{S}_{D^\times}(\mathcal{J}(A))$. When $A$ is admissible, the source and target of this map are also admissible, so to see it is an isomorphism, it suffices to show this after applying $-\otimes_{\mathbf{Z}/p^n}^{\mathbf{L}} \F_p$. But all operations here are compatible with reduction mod $p$, so we now conclude by invoking the duality isomorphism with $\F_p$-coefficients.
\end{rmk}

\subsection{Applications to Gelfand-Kirillov dimension}

In this section we prove Theorems \ref{diminequality} and \ref{dimlowerbound}, along with some related results. We will freely use the basic properties of dimension, in particular the results of Propositions \ref{dimbasics}-\ref{dimzero} and the surrounding discussion.

\begin{proof}[Proof of Theorem \ref{diminequality}]
Let $d=\dim_{G}(\pi)$, so $\mathcal{S}_{G}(\pi)\in \D^{\mathrm{sm}}(G)^{[0,d]}$.
Since $\mathcal{J}(-)$ carries $\D^{\mathrm{sm}}(G)^{[a,b]}$ into
$\D^{\mathrm{sm}}(D^\times)^{[a,b+2n-2]}$, we deduce that 
\[
(\mathcal{J}\circ\mathcal{S}_{G})(\pi)[2n-2]\in \D^{\mathrm{sm}}(D^\times)^{ \leq d}.
\]
By the duality isomorphism $(\mathcal{S}_{D^\times}\circ\mathcal{J})(\pi)\cong(\mathcal{J}\circ\mathcal{S}_{G})(\pi)[2n-2]$,
we get that $M=(\mathcal{S}_{D^\times}\circ\mathcal{J})(\pi)\in \D^{\mathrm{sm}}(D^\times)^{ \leq d}$.
Now look at the fourth quadrant spectral sequence
\[
E_{2}^{i,j}=\mathcal{S}_{D^\times}^{i}(\mathcal{J}^{-j}(\pi))\Rightarrow H^{i+j}(M).
\]
Let $c=\max_{i}\dim_{D^\times}\mathcal{J}^{i}(\pi)$, and let $i_{0}$ be
the least $i$ realizing this maximum. Then $E_{2}^{c,-i_{0}}$ is
$c$-dimensional as an $D^\times$-representation, and as we turn the pages
of the spectral sequence it only interacts with representations of
dimension $<c$ via the incoming differentials. Therefore a nonzero
subquotient survives to the $E_{\infty}$-page, giving that $H^{c-i_{0}}(M)\neq0$.
But $M\in \D^{\mathrm{sm}}(D^\times)^{ \leq d}$, so this implies that $c-i_{0}\leq d$,
and therefore $c\leq d+i_{0}\leq d+N_{\pi}$, as desired.
\end{proof}

We conjecture the following strengthening of Theorem \ref{diminequality}.

\begin{conjecture}\label{diminequalityoptimal} For any $\pi\in\mathrm{Rep}_{\F_p}^{\mathrm{adm}}(G)$, we have $\dim_{D^\times} \mathcal{J}^i(\pi) \leq \dim_{G} \pi$ for all $i\geq 0$.
\end{conjecture}
In many cases, this can be proved. In particular, we have the following result, which can be proved by an easy adaptation of the proof of Theorem \ref{diminequality}.

\begin{prop}Suppose $\pi\in\mathrm{Rep}_{\F_p}^{\mathrm{adm}}(G)$ is such that for all $i>n-1$ and all $j\geq 0$, $\mathcal{J}^i(\pi)=0$ and $\mathcal{J}^i(\mathcal{S}_{G}^j(\pi))=0$. Then Conjecture \ref{diminequalityoptimal} is true for $\pi$.
\end{prop}

The hypotheses of this proposition are satisfied for any representation of the form $\pi=\mathrm{Ind}_{P}^G \sigma$, where $P$ is any parabolic subgroup contained in the standard maximal parabolic $P_{n-1,1}$ and $\sigma$ is any admissible smooth representation of the Levi quotient $P\twoheadrightarrow M$. To see this, note that the main theorem of \cite{johansson-ludwig} shows that $\mathcal{J}^i \circ \mathrm{Ind}_{P}^G = 0$ for all $i>n-1$. It's then enough to observe that for all $j$ we have $\mathcal{S}_G^{j}( \mathrm{Ind}_{P}^G \sigma) = \mathrm{Ind}_{P}^G \delta_P \mathcal{S}_{M}^{j-d}(\sigma)$ for a certain integer $d$ and smooth character $\delta_P$ (use \cite[Theorem 4.7 and Corollary 5.3]{kohlhaase}), so the main theorem of \cite{johansson-ludwig} again applies.

We also observe that Conjecture \ref{diminequalityoptimal} implies that $\mathcal{J}^{2n-2}$ is very small, and can be explicitly computed.
\begin{prop}\label{Jtopdegree}Suppose that Conjecture \ref{diminequalityoptimal} is true. Then for any $\pi\in\mathrm{Rep}_{\F_p}^{\mathrm{adm}}(G)$, we have $\dim_{\F_p} \mathcal{J}^{2n-2}(\pi) < \infty$, and in fact there is an isomorphism \[ \mathcal{J}^{2n-2}(\pi) = \mathcal{S}_{D^\times}^{0}(\mathcal{J}^0(\mathcal{S}_{G}^0(\pi))).\]
\end{prop}

The final isomorphism here implies (assuming Conjecture \ref{diminequalityoptimal}) that $\mathcal{J}^{2n-2}(\pi)$ is nonzero \emph{exactly} when $\pi$ admits a $G$-stable quotient $\pi \twoheadrightarrow \tau$ such that $\dim_{\F_p} \tau < \infty$ and $\tau^{\mathrm{SL}_n (F)} \neq 0$. We also note that the following proof shows that $\mathcal{J}^{2n-2}(\pi) = \mathcal{S}_{D^\times}^{0}(\mathcal{J}^0(\mathcal{S}_{G}^0(\pi)))$ whenever $\dim_{D^\times} \mathcal{J}^{2n-2}(\pi) =0$, independently of Conjecture \ref{diminequalityoptimal}.

\begin{proof}Set $M=(\mathcal{J} \circ \mathcal{S}_G)(\pi)$. By the duality theorem, we get an $E_2$ spectral sequence \[ E_2^{i,j}=\mathcal{S}^{i}_{D^\times}(\mathcal{J}^{-j}(\pi)) \Rightarrow H^{i+j+2n-2}(M).\] Suppose that $\dim_{\F_p} \mathcal{J}^{2n-2}(\pi) = \infty$, i.e. that $c:= \dim_{D^\times} \mathcal{J}^{2n-2}(\pi) > 0$. Then $E_{2}^{c,-(2n-2)}$ is nonzero and $c$-dimensional, and a c-dimensional quotient of it survives to the $E_\infty$-page, showing that $\dim_{D^\times} H^c(M) \geq c$. On the other hand, we have a second spectral sequence $\mathcal{J}^i(\mathcal{S}_{G}^j(\pi))\Rightarrow H^{i+j}(M)$, showing that $H^c(M)$ has a finite filtration whose gradeds are subquotients of $\mathcal{J}^{i}(\mathcal{S}_{G}^{c-i}(\pi))$ for varying $i$. Since $\dim_G \mathcal{S}_{G}^{c-i}(\pi) \leq c-i$, Conjecture \ref{diminequalityoptimal} then implies that $\dim_{D^\times} \mathcal{J}^{i}(\mathcal{S}_{G}^{c-i}(\pi)) \leq c-i < c$ for $i>0$. But when $i=0$, $\dim_{D^\times} \mathcal{J}^{0}(\mathcal{S}_{G}^{c}(\pi))=0$ by the explicit calculation of $\mathcal{J}^0$. Therefore $H^c(M)$ has a finite filtration whose gradeds have dimension $<c$, so $\dim_{D^\times} H^c(M) < c$. This is a contradiction.

For the final statement, note that the two spectral sequences show (unconditionally) that \[\mathcal{S}_{D^\times}^0(\mathcal{J}^{2n-2}(\pi)) = \mathcal{J}^0(\mathcal{S}_G^0 (\pi))\] for all $\pi\in\mathrm{Rep}_{\F_p}^{\mathrm{adm}}(G)$. If $\mathcal{J}^{2n-2}(\pi)$ is zero-dimensional, then $\mathcal{S}_{D^\times}^0(\mathcal{J}^{2n-2}(\pi)) = \mathcal{S}_{D^\times}(\mathcal{J}^{2n-2}(\pi))$ is the full derived dual of $\mathcal{J}^{2n-2}(\pi)$, so the result follows from biduality.
\end{proof}

\begin{proof}[Proof of Theorem \ref{dimlowerbound}]
Let $\check{\pi}$ be the unique nonzero derived dual of $\pi$. The duality isomorphism then simplifies to an isomorphism \[ \mathcal{S}_{D^\times}(\mathcal{J}(\pi))\cong \mathcal{J}(\check{\pi})[2n-2-\dim_G \pi]. \]
Since $\mathcal{J}(\check{\pi})$ is concentrated in degrees $\geq 0$, the right-hand side, and then also the left-hand side, is concentrated in degrees $\geq \dim_G \pi -(2n-2)$. 

If $\mathcal{J}(\pi)$ is zero, there is nothing to prove. Otherwise, let $c=\max_{i}\dim_{D^\times}\mathcal{J}^{i}(\pi)$, and let $i_{0}$ be any $i$ realizing this maximum. Arguing as in the proof of Theorem \ref{diminequality}, we see that the left-hand side has nonzero cohomology in degree $c-i_0$. Therefore $c\geq \dim_G \pi -(2n-2) + i_0 \geq \dim_G \pi -(2n-2)$, giving the desired result.
\end{proof}

\subsection{The case of $\mathrm{GL}_2$}

Now we specialize to the case $n=2$, with $F/\mathbf{Q}_{p}$
still arbitrary. Set $d=[F:\mathbf{Q}_{p}]$. Let $B\subset G=\mathrm{GL}_{2}(F)$
be the upper-triangular Borel subgroup, with $T=F^{\times}\times F^{\times}\subset B$
the diagonal maximal torus. Let $\omega:F^{\times}\to\mathbf{F}_{p}^{\times}$
be the character defined as the composition 
\[
F^{\times}\overset{\mathrm{Nm}_{F/\mathbf{Q}_{p}}}{\longrightarrow}\mathbf{Q}_{p}^{\times}\overset{x\mapsto x|x|}{\longrightarrow}\mathbf{Z}_{p}^{\times}\overset{\mathrm{red}}{\to}\mathbf{F}_{p}^{\times}.
\]
If $\chi_{1}\otimes\chi_{2}:T\to\mathbf{F}_{p}^{\times}$
is any smooth character, we can consider the smooth induction $\mathrm{Ind}_{B}^{G}(\chi_{1}\otimes\chi_{2})$.
This is irreducible iff $\chi_{1}\neq\chi_{2}$. Moreover, $\mathcal{S}_{G}^{i}(\mathrm{Ind}_{B}^{G}(\chi_{1}\otimes\chi_{2}))=0$
for all $i\neq d$, and 
\[
\mathcal{S}_{G}^{d}(\mathrm{Ind}_{B}^{G}(\chi_{1}\otimes\chi_{2}))\cong\mathrm{Ind}_{B}^{G}(\omega\chi_{1}^{-1}\otimes\omega^{-1}\chi_{2}^{-1}).
\]
In particular, if we set $\pi(\chi_{1},\chi_{2})=\mathrm{Ind}_{B}^{G}(\omega\chi_{1}\otimes\chi_{2})$,
then 
\[
\mathcal{S}_{G}^{d}(\pi(\chi_{1},\chi_{2}))\otimes(\chi_{1}\chi_{2}\omega\circ\det)\cong\pi(\chi_{2},\chi_{1}),
\]
and moreover $\pi(\chi_{1},\chi_{2})$ and its dual $\mathcal{S}_{G}^{d}(\pi(\chi_{1},\chi_{2}))$
are both irreducible iff $\chi_{1}\chi_{2}^{-1}\neq\omega^{\pm1}$.

Assume then that $\chi_1,\chi_2$ is a pair of characters such that $\chi_1/\chi_2 \neq \omega^{\pm 1}$. As noted in the introduction, we then have $\mathcal{J}^i(\pi(\chi_1,\chi_2))=0$ for all $i\neq 1$, and likewise with $\chi_1$ and $\chi_2$ swapped. With this vanishing in mind, we set $\tau(\chi_1,\chi_2)=\mathcal{J}^1(\pi(\chi_1,\chi_2))$.

\begin{thm}Notation and assumptions as above, exactly one of the following is true.

i. $\tau(\chi_1,\chi_2)= \tau(\chi_2,\chi_1)=0.$

ii. Both $\tau(\chi_1,\chi_2)$ and $\tau(\chi_2,\chi_1)$ are nonzero, and are exchanged by the functor $\mathcal{S}_{D^\times}^d(-) \otimes (\chi_1 \chi_2 \omega \circ \mathrm{Nm})$. Moreover, $\mathcal{S}_{D^\times}^i(-)$ applied to either representation vanishes for all $i\neq d$.
\end{thm}

\begin{proof}

Evaluate the duality isomorphism $(\mathcal{S}_{D^\times}\circ\mathcal{J})\cong(\mathcal{J}\circ\mathcal{S}_{G})[2]$
 on the representation $\pi(\chi_1,\chi_2)$. Since $\mathcal{J}(\pi(\chi_1,\chi_2)) = \tau(\chi_1,\chi_2)[-1]$, the left-hand side simplifies to $\mathcal{S}_{D^\times}(\tau(\chi_1,\chi_2))[1]$ by the results recalled above. On the other hand, $\mathcal{S}_{G}(\pi(\chi_{1},\chi_{2})) \cong\pi(\chi_{2},\chi_{1}) \otimes (\chi_{1}\chi_{2}\omega\circ\det)^{-1}[-d]$, so the right-hand side simplifies to $\mathcal{J}(\pi(\chi_{2},\chi_{1}))\otimes(\chi_{1}\chi_{2}\omega\circ\mathrm{Nm})^{-1}[2-d] \cong \tau(\chi_{2},\chi_{1})\otimes(\chi_{1}\chi_{2}\omega\circ\mathrm{Nm})^{-1}[1-d]$. Thus, canceling the shifts, we get an isomorphism 
\[\mathcal{S}_{D^\times}(\tau(\chi_1,\chi_2)) \cong \tau(\chi_{2},\chi_{1})\otimes(\chi_{1}\chi_{2}\omega\circ\mathrm{Nm})^{-1}[-d],
\]which immediately gives the desired result.
\end{proof}

\subsubsection{A non-generic case}\label{nongeneric}
In this section we sketch some consequences of our duality theorem in the setting of \emph{reducible} principal series representations of $\mathrm{GL}_2(\mathbf{Q}_p)$. Maintain the previous notation, but with $F=\mathbf{Q}_p$. The principal series representation $\pi(\omega^{-1},1)$ is reducible of length two, and sits in a short exact sequence \[0 \to \mathbf{1} \to \pi(\omega^{-1},1) \to \mathrm{St} \to 0\;\;\;\;\; (\dagger)\]
where $\mathrm{St}$ is the mod $p$ Steinberg representation of $\mathrm{GL}_2(\mathbf{Q}_p)$. To ease notation, set $\kappa =\mathcal{S}^1_{G}(\mathrm{St})$. Note that $\mathcal{S}_{G}(\mathbf{1})=\mathbf{1}[0]$ and $\mathcal{S}_{G}(\pi(\omega^{-1},1))=\pi(1,\omega^{-1})[-1]$, where $\pi(1,\omega^{-1})$ is an irreducible principal series.  The associated long exact sequence for $\mathcal{S}^{i}_{G}(-)$ is now easy to analyze: the embedding $\mathcal{S}^0_{G}(\mathrm{St})\hookrightarrow \mathcal{S}^0_{G}(\pi(\omega^{-1},1))=0$ implies that the long exact sequence degenerates to a short exact sequence
\[ 0 \to \mathbf{1} \to \kappa \to \pi(1,\omega^{-1}) \to 0\;\;\;\;(\ddagger)\]
and that $\mathcal{S}^i_{G}(\mathrm{St})$ for all $i\neq 1$. Biduality then implies that the previous sequence is non-split, since otherwise the nonvanishing of $\mathcal{S}^0_{G}(-)$ applied to the first term would contradict the vanishing of $\mathcal{S}^0_{G}(\kappa)$. Note that $\kappa$ is the unique non-split extension of $\pi(1,\omega^{-1})$ by $\mathbf{1}$.

Using the main result in \cite{ludwig-quotient} together with the explicit computation of $\mathcal{J}^0$, the following vanishing results are clear:
\begin{itemize}

\item $\mathcal{J}(\mathbf{1})=R\Gamma(\mathbf{P}^1_{C},\F_p)=\mathbf{1} \oplus \mathbf{1}[-2]$

    \item $\mathcal{J}^i(\pi(1,\omega^{-1}))=0$ for all $i\neq 1$.

\item $\mathcal{J}^i(\pi(\omega^{-1},1))=0$ for all $i\neq 0,1$, and $\mathcal{J}^0(\pi(\omega^{-1},1))=\mathbf{1}$.

\item $\mathcal{J}^i(\mathrm{St})=0$ for all $i \neq 1$.

\end{itemize}
The long exact sequence in $\mathcal{J}^i$'s associated with the sequence $(\dagger)$ then degenerates to a short exact sequence
\[0 \to \mathcal{J}^1(\pi(\omega^{-1},1)) \to \mathcal{J}^{1}(\mathrm{St}) \to \mathbf{1} \to 0\;\;\;\;\;(\dagger').\]
Next, observe that $\mathcal{J}^2$ is right-exact, so applying it to the exact sequence $(\ddagger)$ gives a surjection $\mathbf{1}=\mathcal{J}^2(\mathbf{1})\twoheadrightarrow \mathcal{J}^2(\kappa)$. This shows that $\dim_{D^\times} \mathcal{J}^2(\kappa) =0$. Using Proposition \ref{Jtopdegree} and the remark immediately afterwards, we then deduce from the vanishing $\mathcal{S}^0_{G}(\kappa)=0$ that in fact $\mathcal{J}^2(\kappa)=0$. Note that Hu-Wang proved the vanishing $\mathcal{J}^2(\kappa)=0$ by global methods \cite[Corollary 8.19]{hu-wang}, while our argument is purely local. With this vanishing in hand, we see that the long exact sequence in $\mathcal{J}^i$'s associated with the sequence $(\ddagger)$ degenerates to a short exact sequence
\[ 0 \to \mathcal{J}^{1}(\kappa) \to \mathcal{J}^1(\pi(1,\omega^{-1})) \to \mathbf{1} \to 0\;\;\;\;\;(\ddagger'). \]

Next, we observe that the representations $\mathrm{St}$ and $\pi(1,\omega^{-1})$ both satisfy the conditions of Proposition \ref{dualitysimple},  with $c=1$ and $\dim_G(-) = 1$ in both cases. Applying that proposition, we get duality isomorphisms \[\mathcal{S}^{i}_{D^\times}(\mathcal{J}^{1}(\mathrm{St}))\cong \mathcal{J}^{i}(\kappa)\]
and
\[\mathcal{S}^{i}_{D^\times}(\mathcal{J}^{1}(\pi(1,\omega^{-1})))\cong \mathcal{J}^{i}(\pi(\omega^{-1},1))
\]
for all $i\geq 0$. When $i=0$ these yield no new information. The vanishing of the right-hand sides for $i \geq 2$ implies that $\mathcal{J}^{1}(\mathrm{St})$ and $\mathcal{J}^{1}(\pi(1,\omega^{-1}))$ are of dimension $\leq 1$; the exact sequences $(\dagger')$ and $(\ddagger')$ then imply that $\mathcal{J}^{1}(\kappa)$ and $\mathcal{J}^{1}(\pi(\omega^{-1},1))$ are also of dimension $\leq 1$. Finally, in the case $i=1$, the left-hand side of each isomorphism is either identically zero, or Cohen-Macaulay of dimension one. This yields a dichotomy: either
\begin{enumerate}
    \item $\mathcal{J}^1(\mathrm{St}) = \mathcal{J}^1(\pi(1,\omega^{-1}))=\mathbf{1}$ and $\mathcal{J}^{1}(\pi(\omega^{-1},1))=\mathcal{J}^{1}(\kappa)=0$, or
    \item all four representations $\mathcal{J}^1(\mathrm{St})$, $\mathcal{J}^1(\pi(1,\omega^{-1}))$, $\mathcal{J}^{1}(\pi(\omega^{-1},1))$, and $\mathcal{J}^{1}(\kappa)$ are one-dimensional, and the last two are Cohen-Macaulay.
\end{enumerate}
Just as in the situation of Theorem \ref{principalseriesdichotomy}, we don't currently see how to rule out (1) by a purely local argument. However, using \cite[Proposition 8.18.(i)]{hu-wang} and their exact sequence (8.13), it is easy to see that (1) does not occur. We leave the details to the interested reader.

\section{A partial K\"unneth formula}\label{partialkunnethsection}
In this section only, we change notation slightly, to emphasize the dependence of the Jacquet-Langlands functor on the integer $n$. More precisely, we write $\mathcal{J}_n$ for the Jacquet-Langlands functor from $\D(\mathrm{GL}_n(F))$ towards $\D(D_{1/n}^{\times})$.

For any integer $0<d<n$, let $P_{n-d,d}\subset \mathrm{GL}_n(F)$ be the usual upper-triangular parabolic with Levi $\mathrm{GL}_{n-d}(F)\times \mathrm{GL}_d(F)$. Let $\sigma_1 \boxtimes \sigma_2$ be a smooth representation of the Levi, and let $\mathrm{Ind}_{P_{n-d,d}}^{\mathrm{GL}_n(F)}(\sigma_1 \boxtimes \sigma_2)$ be the usual parabolic induction to a smooth representation of $\mathrm{GL}_n(F)$. The main result in this section is an alternative formula for $\mathcal{J}_n(\mathrm{Ind}_{P_{n-d,d}}^{\mathrm{GL}_n(F)}(\sigma_1 \boxtimes \sigma_2))$, in terms of the cohomology of an auxiliary space with coefficients in a certain sheaf built from $\sigma_1$ and $\mathcal{J}_d(\sigma_2)$.  This builds on the analysis begun by the first author in \cite[Appendix A]{johansson-ludwig}. When $d=1$, this formula is implicit in \cite{johansson-ludwig}, as explained below. 

To state the main result, we need to introduce the key auxiliary space. In what follows we will freely use some standard notation related to relative Fargues-Fontaine curves. In particular, for any perfectoid space $S/\Spa C$, we have the adic relative Fargues-Fontaine curve $\mathcal{X}_{S^\flat}$, which comes with a canonical closed immersion $\iota: S \hookrightarrow \mathcal{X}_{S^\flat}$ realizing $S$ as a closed Cartier divisor. We write $\mathcal{O}(1/n)$ for the usual stable vector bundle of rank $n$ and degree $1$, which is defined functorially over any $\mathcal{X}_{S^\flat}$. 

\begin{defn}Let $W_{n,d}$ denote the moduli space over $\Spd C$ parametrizing quotients of $\mathcal{O}(1/n)$ which are isomorphic to $\mathcal{O}(1/d)$ after pullback to any geometric point.
\end{defn}

Note that $W_{n,d}$ carries a natural $\underline{D_{1/n}^\times}$-action by bundle automorphisms.\footnote{Recall that the automorphism sheaf $S \mapsto \mathrm{Aut}(\mathcal{O}(1/n)^{\oplus m} / \mathcal{X}_{S^\flat})$ is exactly the sheaf $\underline{\mathrm{GL}_m(D_{1/n}^{\times})}$.} It is easy to see that $W_{n,d}$ is a spatial diamond, and we will see below that it is proper over $\Spd C$. It is also easy to see that $W_{n,d}$ can be defined equivalently as the moduli space of \emph{sub}bundles of $\mathcal{O}(1/n)$ which are isomorphic to $\mathcal{O}^{n-d}$ after pullback to any geometric point.  Trivializing the subbundle in this alternative description  gives a canonical $\mathrm{GL}_{n-d}(F)$-torsor over $W_{n,d}$, corresponding to a canonical map $h_1: W_{n,d} \to B \underline{\mathrm{GL}_{n-d}(F)}$. On the other hand, trivializing the quotient bundle in the original moduli problem gives a canonical $D_{1/d}^\times$-torsor over $W_{n,d}$, corresponding to a canonical map $h_2: W_{n,d} \to B \underline{D_{1/d}^\times}$. 

With these preparations in hand, the main result in this section can be stated as follows.

\begin{thm}\label{partialkunneth} For any $\sigma_1 \in \D(\mathrm{GL}_{n-d}(F))$ and $\sigma_2 \in \D(\mathrm{GL}_{d}(F))$, there is a natural isomorphism
\[\mathcal{J}_n(\mathrm{Ind}_{P_{n-d,d}}^{\mathrm{GL}_n(F)}(\sigma_1 \boxtimes \sigma_2)) \cong R\Gamma(W_{n,d},h_{1,\et}^\ast \sigma_1 \otimes h_{2,\et}^\ast \mathcal{J}_d(\sigma_2)).
\]
\end{thm}

When $d=1$, this specializes to the isomorphism $\mathcal{J}_{n}(\mathrm{Ind}_{P_{n-1,1}}^{\mathrm{GL}_n(F)}(\sigma_1 \boxtimes \sigma_2)) \cong R\Gamma(\mathcal{M}_{P(F)}, \mathcal{F}_{\sigma_1 \boxtimes \sigma_2})$ implicit in the proof of \cite[Theorem 5.3.1]{johansson-ludwig} (where we freely use some notation from loc. cit.). Indeed, $\mathcal{M}_{P(F)} = X_{n,1} = W_{n,1}$ by the discussion below, and a direct calculation shows that $\mathcal{F}_{\sigma_1 \boxtimes \sigma_2} \cong h_1^\ast \sigma_1 \otimes h_2^\ast \sigma_2$, noting that $\mathcal{J}_1(\sigma_2)=\sigma_2$ has no effect.

The proof relies on several auxiliary spaces, which we find helpful to collect into a large diagram
\[
\xymatrix{X\ar[r]\ar[d] & X_{-}\ar[d]\ar[r] & [\mathcal{M}_{\infty,d}/\underline{D_{1/d}^{\times}}]\ar[r]\ar[d] & \ast\ar[d]\\
X_{+}\ar[r]\ar[d] & X_{n,d}\ar[d]^{\tilde{g}_d}\ar[r]^{\tilde{h}_{2}} & [\mathbf{P}_{C}^{d-1}/\underline{D_{1/d}^{\times}}]\ar[r]^{f_d}\ar[d]^{g_d} & B\underline{\mathrm{GL}_{d}(F)}\\
\mathcal{I}\mathrm{nj}(\mathcal{O}^{n-d},\mathcal{O}(1/n))\ar[d]\ar[r] & W_{n,d}\ar[d]^{h_{1}}\ar[r]^{h_{2}} & B\underline{D_{1/d}^{\times}}\\
\ast\ar[r] & B\underline{\mathrm{GL}_{n-d}(F)}
}
\]
of small v-stacks, where every square is cartesian. Here $\ast = \Spd C$ for brevity, and $W_{n,d}$ and $h_1,h_2$ are as discussed above. Let us explain the parts of this diagram we have not yet discussed.\footnote{We will not need every object in this diagram, but the reader might find it helpful to know that they exist.} Let $U\subset P_{n-d,d}$ be the unipotent radical, so $U\cdot \mathrm{GL}_{n-d}(F)$ and $U\cdot \mathrm{GL}_{d}(F)$ are both closed normal subgroups of $P_{n-d,d}$ with intersection $U$. Then by definition, $X=\mathcal{M}_{\infty}/\underline{U}$, $X_+ = \mathcal{M}_{\infty}/\underline{U\cdot \mathrm{GL}_{d}(F)}$, $X_- = \mathcal{M}_{\infty}/\underline{U\cdot \mathrm{GL}_{n-d}(F)}$, and $X_{n,d} = \mathcal{M}_\infty / \underline{P_{n-d,d}}$, whence the upper left square is cartesian. By \cite[Proposition A.3.3]{johansson-ludwig}, the space $X_{n,d}$ admits an explicit alternative description: it is the moduli space sending any perfectoid space $S/C$ to (isomorphism classes of) diagrams $\mathcal{O}(1/n)\twoheadrightarrow \mathcal{E} \hookleftarrow \mathcal{F}$ of vector bundles over $\mathcal{X}_{S^\flat}$ such that $\mathcal{E} \simeq \mathcal{O}(1/d)$ and $\mathcal{F} \simeq \mathcal{O}^d$ at all geometric points, and such that $\mathrm{coker}(\mathcal{F} \to \mathcal{E}) \simeq \iota_{\ast}W$ for some rank one projective $\mathcal{O}_S$-module $W$, where $\iota:S \to \mathcal{X}_{S^\flat}$ is the usual divisor. The map $\tilde{h}_2$ corresponds to remembering only $\mathcal{E} \hookleftarrow \mathcal{F}$, while $\tilde{g}_d$ corresponds to remembering only $\mathcal{O}(1/n)\twoheadrightarrow \mathcal{E}$. The cartesian square spanned by $g_d$ and $h_2$ is then a consequence of \cite[Proposition A.3.4]{johansson-ludwig}.  Note that $g_d$ is proper and $p$-bounded, since it comes from a proper map of rigid analytic spaces pro-\'etale-locally on the target. Note also that $X_{n,d} \to W_{n,d}$ is proper and surjective, while $X_{n,d} \to \ast$ is proper by \cite[Proposition A.3.2]{johansson-ludwig}, so $W_{n,d} \to \ast$ is proper as claimed in the discussion preceding Theorem \ref{partialkunneth}.

\begin{lem}For any $M_1 \in \D^{\sm,a}_\solid(\mathrm{GL}_{n-d}(F),\ri_C/p)^{\varphi}$ and $M_2 \in \D^{\sm,a}_\solid(\mathrm{GL}_{d}(F),\ri_C/p)^{\varphi}$,  there is a natural isomorphism
\[\mathcal{J}_{n,\solid}(\mathrm{Ind}_{P_{n-d,d}}^{\mathrm{GL}_n(F)}(M_1 \boxtimes M_2)) \cong R\Gamma_{\solid}(X_{n,d}, \tilde{g}_d^\ast h_1^{\ast} M_1 \otimes \tilde{h}_{2}^\ast f_d^{\ast} M_2) \]
in $\D^{\sm,a}_\solid(D^\times, \ri_C/p)^\varphi$.
\end{lem}
\begin{proof}
Consider the commutative diagram
\begin{center}\begin{tikzcd}
{X_{n,d}}\ar[d, "\tilde{\alpha}"] \arrow[rr] & & {\ast} \ar[d, "\alpha"] \\
    {[X_{n,d}/\underline D_{1/n}^\times]} \arrow[r,"\overline\pi_{\mathrm{GH}}"] \arrow[d,"r"] & {[\mathbf P_C^{n-1}/\underline D_{1/n}^\times]} \arrow[d,"f_n"] \arrow[r,"g_n"] & B\underline D_{1/n}^\times\\
    B\underline{P_{n-d,d}} \arrow[r,"q"] \arrow[d,"s"] & B\underline{\GL_n(F)}\\
    B\underline{\GL_{n-d}(F)} \times B\underline{\GL_d(F)}
\end{tikzcd}\end{center}
where the square and rectangle are both cartesian.
By definition, $\mathrm{Ind}_{P_{n-d,d}}^{\mathrm{GL}_n(F)}(M_1 \boxtimes M_2) = q_{\ast} s^{\ast}(M_1 \boxtimes M_2)$ and $\mathcal{J}_{n,\solid} = g_{n \ast} f_{n}^{\ast}$, so
\begin{align*}
    \mathcal{J}_{n,\solid}(\mathrm{Ind}_{P_{n-d,d}}^{\mathrm{GL}_n(F)}(M_1 \boxtimes M_2)) &= g_{n \ast} f_n^{\ast} q_{\ast} s^{\ast}(M_1 \boxtimes M_2),
\intertext{which by proper base-change transforms to}
    &= g_{n*} (\overline\pi_{\mathrm{GH}})_* r^* s^* (M_1 \boxtimes M_2)\\
    &= R\Gamma_\solid(X_{n,d}, r^* s^* (M_1 \boxtimes M_2)),
\end{align*}
where the $R\Gamma_\solid$ on the last line denotes the $D_{1/n}^\times$-equivariant cohomology. To finish the proof it remains to identify the $D_{1/n}^\times$-equivariant sheaf $r^* s^* (M_1 \boxtimes M_2)$ on $X_{n,d}$. After forgetting the $D_{1/n}^\times$-equivariance of the sheaf, i.e. after applying $\tilde{\alpha}^*$, we get
\begin{align*}
    \tilde{\alpha}^* r^* s^* (M_1 \boxtimes M_2) &= (s \circ r \circ \tilde{\alpha})^* (M_1 \boxtimes M_2) = ((h_1 \circ \tilde{g}_d) \times (f_d \circ \tilde{h}_2))^* (M_1 \boxtimes M_2)\\
    &= \tilde{g}_d^\ast h_1^{\ast} M_1 \otimes \tilde{h}_{2}^\ast f_d^{\ast} M_2,
\end{align*}
using that $s\circ r \circ \tilde{\alpha} = (h_1 \circ \tilde{g}_d) \times (f_d \circ \tilde{h}_2)$. This gives the result.
\end{proof}

\begin{proof}[Proof of Theorem \ref{partialkunneth}]
By some straightforward manipulations with the Riemann-Hilbert correspondence, we easily reduce to proving the solid version of the theorem. In other words, we need to show that for any $M_1 \in \D^{\sm,a}_\solid(\mathrm{GL}_{n-d}(F),\ri_C/p)^{\varphi}$ and $M_2 \in \D^{\sm,a}_\solid(\mathrm{GL}_{d}(F),\ri_C/p)^{\varphi}$,  there is a natural isomorphism
\[\mathcal{J}_{n,\solid}(\mathrm{Ind}_{P_{n-d,d}}^{\mathrm{GL}_n(F)}(M_1 \boxtimes M_2)) \cong R\Gamma_{\solid}(W_{n,d},h_1^\ast M_1 \otimes h_2^\ast \mathcal{J}_{d,\solid}(M_2)).
\]
By the previous lemma, we already have a canonical isomorphism \[\mathcal{J}_{n,\solid}(\mathrm{Ind}_{P_{n-d,d}}^{\mathrm{GL}_n(F)}(M_1 \boxtimes M_2)) \cong R\Gamma_{\solid}(X_{n,d}, \tilde{g}_d^\ast h_1^{\ast} M_1 \otimes \tilde{h}_{2}^\ast f_d^{\ast} M_2). \]  We then compute that
\begin{align*}
    R\Gamma_{\solid}(X_{n,d}, \tilde{g}_d^\ast h_1^{\ast} M_1 \otimes \tilde{h}_{2}^\ast f_d^{\ast} M_2) & \cong R\Gamma_{\solid}(W_{n,d}, \tilde{g}_{d\ast}(\tilde{g}_d^\ast h_1^{\ast} M_1 \otimes \tilde{h}_{2}^\ast f_d^{\ast} M_2)) \\
      &\cong R\Gamma_{\solid}(W_{n,d}, h_1^{\ast} M_1 \otimes \tilde{g}_{d\ast} \tilde{h}_{2}^\ast f_d^{\ast} M_2)) \\
      &\cong R\Gamma_{\solid}(W_{n,d}, h_1^{\ast} M_1 \otimes h_{2}^\ast g_{d\ast} f_d^{\ast} M_2)).
\end{align*}
Here the first isomorphism is trivial, the second follows from the projection formula, and the third follows from proper base change. Since $g_{d\ast} f_d^{\ast} M_2= \mathcal{J}_{d,\solid}(M_2)$ by definition, this gives the desired result.
\end{proof}

By work of Johansson-Ludwig \cite[Theorem A]{johansson-ludwig}, $W_{n,1}$ is a perfectoid space for all $n$. The proof of \cite[Theorem A]{johansson-ludwig} is a highly nontrivial argument using the $p$-adic geometry of global Shimura varieties. In the special case $n=2$, the first author reproved this result by purely local methods \cite[\S A.4]{johansson-ludwig}. However, the arguments in \cite[\S A.4]{johansson-ludwig} easily generalize to show that $W_{n,n-1}$ is also a perfectoid space for all $n$. It thus seems natural to make the following conjecture.

\begin{conjecture}\label{auxiliaryperfectoid}For all integers $0<d<n$, $W_{n,d}$ is a perfectoid space.
\end{conjecture}

This conjecture implies a vanishing result.

\begin{prop}Suppose $W_{n,d}$ is a perfectoid space. Then for any admissible smooth representations $\sigma_1 \in \mathrm{Rep}_{\mathbf{F}_{p}}^{\mathrm{sm}}(\mathrm{GL}_{n-d}(F))$ and $\sigma_2 \in \mathrm{Rep}_{\mathbf{F}_{p}}^{\mathrm{sm}}(\mathrm{GL}_{d}(F))$, $\mathcal{J}_{n}^{i}(\mathrm{Ind}_{P_{n-d,d}}^{\mathrm{GL}_n(F)}(\sigma_1 \boxtimes \sigma_2))=0$ for all $i>d+n-2$. 
\end{prop}
Note that $\mathcal{J}_{n}^{i}$ trivially vanishes in degrees $>2n-2$, while this result gives vanishing in degrees $> d+n-2=2n-2-(n-d)$, so we improve on the trivial bound by an interval of length $n-d$. When $d=1$, this specializes to the main vanishing theorem in \cite{johansson-ludwig} (using of course the perfectoidness of $W_{n,1}$ proved in loc. cit.).

\begin{proof}By the primitive comparison theorem plus some small additional arguments, we get an almost isomorphism
\[\mathcal{J}_{n}^{i}(\mathrm{Ind}_{P_{n-d,d}}^{\mathrm{GL}_n(F)}(\sigma_1 \boxtimes \sigma_2)) \otimes \mathcal{O}_C / p \cong H^i(W_{n,d}, h_1^{\ast} \sigma_1 \otimes h_{2}^\ast \mathcal{J}_d( \sigma_2) \otimes \mathcal{O}^+/p).
\]
It thus suffices to show that $H^i(W_{n,d}, h_1^{\ast} \sigma_1 \otimes h_{2}^\ast \mathcal{J}_d(\sigma_2) \otimes \mathcal{O}^+/p)$ is almost zero in the claimed range of degrees. Recall that $\mathcal{J}_d(\sigma_2)$ is concentrated in degrees $\leq 2d-2$. Note also that $W_{n,d}$ is perfectoid by assumption, so the functor $H^i(W_{n,d}, - \otimes \mathcal{O}^+/p)$ on overconvergent \'etale $\F_p$-sheaves is almost zero for all $i> \dim W_{n,d}=n-d$, by the usual almost equality of \'etale and analytic cohomology. The evident spectral sequence then gives almost vanishing in degrees $> 2d-2 + n-d=d+n-2$, whence the claim.
\end{proof}

When $n=3$, $W_{3,1}$ and $W_{3,2}$ are both perfectoid spaces. Combining this observation with the previous proposition, the right-exactness of $\mathcal{J}_{3}^{4}$, and the classification of irreducible admissible representations of $\mathrm{GL}_3(F)$ \cite[Theorem 1.1]{herzig-classification}, we get the following vanishing result.

\begin{cor}Let $\pi $ be an irreducible admissible representation of $\mathrm{GL}_3(F)$ which is not supersingular, and whose underlying vector space is infinite-dimensional. Then $\mathcal{J}^{4}(\pi)=0$.
\end{cor}

\printbibliography

\end{document}